\documentclass[11pt]{amsart}
\usepackage{amsmath,amsfonts,amssymb,amsthm,epsfig,verbatim,ifthen,graphicx}
\usepackage[usenames,dvips]{color}
\usepackage[all]{xy}
 \numberwithin{equation}{section}

\theoremstyle{definition} 
\newtheorem{definition}{Definition}[section]
\newtheorem{theorem}[definition]{Theorem}
\newtheorem{proposition}[definition]{Proposition}
\newtheorem{lemma}[definition]{Lemma}

\theoremstyle{definition}
\newtheorem{coro}[definition]{Corollary}

\newtheorem{remark}[definition]{Remark}

\newcommand{\N}{\mathbb{N}}

\newcommand{\R}{\mathbb{R}}

\def\F{{\mathcal F}}
\def\G{{\mathcal G}}
\def\H{{\mathcal H}}
\begin{document}

\title
{On factors of Gibbs measures for almost additive potentials}
\author{Yuki Yayama}
\address{Departamento de Ciencias B\'{a}sicas, Universidad del B\'{i}o-B\'{i}o, Av.Andr\'{e}s Bello, s/n Casilla 447
Chill\'{a}n, Chile} \email{yyayama@ubiobio.cl}
\begin{abstract}
Let $(X, \sigma_X), (Y, \sigma_Y)$ be one-sided subshifts and $\pi:X\rightarrow Y$ 
a factor map. Suppose that $X$ has the specification property.
Let $\mu$ be a unique invariant Gibbs measure for a sequence of continuous functions 
$\F=\{\log f_n\}_{n=1}^{\infty}$ on $X$, which is an almost additive potential with bounded variation. 
We show that $\pi\mu$ is a unique invariant Gibbs measure for a sequence of continuous functions 
$\G=\{\log g_n\}_{n=1}^{\infty}$ on $Y$.
When $(X, \sigma_X)$ is a full shift, we characterize $\G$ and $\mu$ by using relative pressure. 
This $\G$ is a generalization of a continuous function found by Pollicott and Kempton in 
their work on factors of Gibbs measures for continuous functions.  
We also consider the following question: Given a unique invariant Gibbs measure 
$\nu$ for a sequence of continuous functions $\F_2$ on $Y$, can we find an invariant Gibbs measure 
$\mu$ for  a sequence of continuous functions $\F_1$ on $X$   
such that $\pi\mu=\nu$? We show that such a measure exists under a certain condition. 
In particular, if $(X, \sigma_X)$ is a full shift and $\nu$ is a unique invariant 
Gibbs measure for a function in the Bowen class, then there exists a preimage $\mu$ of~$\nu$ 
which is a unique invariant Gibbs measure for a function in the Bowen class.
\end{abstract}
\maketitle
\pagestyle{myheadings}
\markright{On factors of Gibbs measures for almost additive potentials}
\section{Introduction}\label{intro}

Let $(X, \sigma_X), (Y, \sigma_Y)$ be one-sided shifts of finitely many symbols and $\pi:X \rightarrow Y$ a factor map.
A factor map $\pi$ is a continuous and surjective function that satisfies $\pi\circ \sigma_X=\sigma_Y\circ \pi$. 
We have the following general questions concerning factors of Gibbs measures.
Given an invariant Gibbs measure $\mu$ for a continuous function $f$ on $X$,   
what are the properties of the image $\pi\mu$ of $\mu$ under $\pi$? Under what conditions is $\pi\mu$ an invariant Gibbs measure for a continuous 
function $g$
on $Y$? What properties must $g$ have?   
For a survey of the study of factors of Gibbs measures for continuous functions, see the paper by Boyle and Petersen \cite{BP}.
For more results on this topic, see \cite{banff}. 
Recently, problems on factors of Gibbs measures for functions of summable variation
have been studied \cite{CUO, CU, V, JY, PK, K}. In particular, it is known from Chazottes and Ugalde \cite{CU} that if 
$\mu$ is a 
unique invariant Gibbs measure for a H\"older continuous function $f$ on $X$, where $X$ is a full shift,  
then $\pi\mu$ is a unique invariant Gibbs measure for a continuous function. 
Pollicott and Kempton \cite{PK}
showed the related results, namely, if $\mu$ is a unique invariant Gibbs measure for a function on $X$ of summable variation with a certain condition, 
then $\pi\mu$ is a unique invariant Gibbs measure for a function on $Y$ of summable variation.
Kempton \cite{K} also extended the results of \cite{PK} to the case when $X$ is a topologically mixing shift of finite type.       

On the other hand, a theory of equilibrium states for sequences of continuous functions has recently been developed \cite{B2006, m2006, 
CFH, FH, CZ, ZC, Feng, Y3, IY}. 
This extends the pressure theory for continuous functions (see \cite{W1}). 
The results have been applied 
in dimension theory in dynamics.  
In particular, the Gibbs measures for sequences of continuous functions have been useful for studying  
non-conformal repellers \cite{B2006, Feng, FH, Y2, Y3}. 

Pressure theory for sequences of continuous functions called almost additive potentials 
(see page \pageref{defaa} for definition of an almost additive potential) 
has been studied by Barreira \cite{B2006} and Mummert \cite{m2006}. 
Almost additive potentials are a generalization of continuous functions
that belong to the Bowen class.
Barreira and Mummert defined topological pressure for an almost additive potential $\F=\{\log f_n\}_{n=1}^{\infty}$, 
proved the variational principle, and studied equilibrium states.
Similarly, Cao, Feng and Huang \cite{CFH} studied pressure theory for  
subadditive potentials (see page \pageref{defaa} for definition) while Feng and Huang \cite{FH} studied it 
in the context of asymptotically subadditive potentials (see page \pageref{defas} for definition). Asymptotically subadditive potentials
generalize almost additive potentials and subadditive potentials. 
The notion of a Gibbs measure for a continuous function was also generalized to that of a Gibbs measure for a sequence of continuous 
functions \cite{B2006, m2006}.  

In this paper, using theory of equilibrium states for sequences of continuous functions, 
we study the image of an invariant Gibbs measure for a sequence of continuous functions under a factor map.
In particular, we consider the image of an invariant Gibbs measure for an almost additive potential.
This generalizes theory of factors of Gibbs equilibrium states for continuous functions (see section \ref{sectioncomp}).

Let $(X, \sigma_X), (Y, \sigma_Y)$ be one-sided subshifts and $\pi:X \rightarrow Y$ a factor map. Suppose that $X$ has the specification property.
In Section \ref{sectionfactor}, we consider the image of an invariant Gibbs measure $\mu$ for an almost additive potential 
$\F=\{\log f_n\}_{n=1}^{\infty}$ on $X$ with bounded variation. Our main question is the following: 
Can we find a sequence of continuous functions $\G=\{\log g_n\}_{n=1}^{\infty}$
on $Y$ such that $\pi\mu$ is an invariant Gibbs measure 
for $\G$?  
We will answer the question by showing in Theorem \ref{main1} that the image $\pi\mu$ is a unique invariant Gibbs measure 
for an 
asymptotically subadditive potential 
$\G=\{\log g_n\}_{n=1}^{\infty}$ on $Y$ with bounded variation. It is a unique equilibrium state for 
$\G$ and it is mixing. In particular, in Corollary \ref{maincoro}, we prove that if an invariant measure $\mu$ is 
an invariant Gibbs measure for a continuous function $f$ on $X$ that belongs to
the Bowen class (we use the notation $Bow (X)$ for the Bowen class), 
then the image $\pi\mu$ is a unique invariant Gibbs measure for a subadditive potential $\bar \G=\{\log \bar g_n\}_{n=1}^{\infty}$ on $Y$ with bounded variation. 

Let $(X, \sigma_X), (Y, \sigma_Y)$ be one-sided sofic shifts and $\pi:X \rightarrow Y$ a factor map. Suppose that $X$ has the specification property.
In Section \ref{sectionrelativepressure}, our main question is 
to characterize $\mu$ on $X$ using the image measure $\pi\mu$ on $Y$. We study the relation 
between $\mu$ and $\pi\mu$ in connection to relative pressure. 
In Proposition \ref{bow2}, we show that if $\mu$ is an invariant Gibbs measure for $f\in Bow(X)$ on a 
topologically mixing shift of finite type $X$, then we can characterize 
$\pi \mu$ as the unique equilibrium state for the relative pressure $P(\sigma_X, \pi, f)$. Hence we can 
replace the subadditive potential 
$\bar \G$ found in Corollary \ref{maincoro} by  $P(\sigma_X, \pi, f)$. Moreover, 
$\mu$ is a unique relative equilibrium state of $f$ over $\pi\mu$. We generalize this result to the case in which $\mu$ is a 
unique invariant Gibbs measure for an almost 
additive potential $\F=\{\log f_n\}_{n=1}^{\infty}$ on a full shift $X$ with bounded variation.  
In this case $\pi\mu$ is a unique equilibrium state for the relative pressure 
$P(\sigma_X, \pi, \F)$, replacing $\G=\{\log g_n\}_{n=1}^{\infty}$ in Theorem \ref{main1} by $P(\sigma_X, \pi, \F)$, and 
 $\mu$ is a unique relative equilibrium state 
of $\F$ over $\pi\mu$ (see Theorem \ref{general}). 

In Section \ref{sectionpreimage}, we study preimages of Gibbs measures. 
Let $(X,\sigma_X), (Y, \sigma_Y)$ be one-sided sofic shifts. Suppose that $X$ has the specification property. Let $\nu\in M(Y, \sigma_Y)$ be 
the unique invariant Gibbs measure for an almost additive potential 
$\Psi=\{\log \psi_n\}_{n=1}^{\infty}$ on $Y$ with 
bounded variation. We will study the following question: 
Is there any Gibbs measure $\mu\in M(X, \sigma_X)$ associated to a sequence of continuous functions 
$\Phi=\{\log \phi_n\}_{n=1}^{\infty}$ on $X$  
such that $\pi\mu=\nu$?
We show in Theorem \ref{preimage} that, in general, under a certain condition on the factor map $\pi$, 
we can find a unique invariant Gibbs measure $\mu$ for an 
almost additive potential $\Phi=\{\log \phi_n\}_{n=1}^{\infty}$ on $X$ with 
bounded variation such that $\pi\mu=\nu$.   
In particular, if $(X, \sigma_X)$ is a full shift and $\nu$ is a unique invariant Gibbs measure for 
a function $\psi\in Bow(Y)$, then there exists a unique invariant Gibbs measure $\mu$ for $\phi\in Bow(X)$ such 
that $\pi\mu=\nu$ (Corollary \ref{cover1}).
Further, we will investigate how to construct a sequence of continuous functions $\Phi$ on $X$ so that a unique invariant Gibbs measure $\mu$ 
associated to $\Phi$ is a 
preimage of $\nu$. To this end, we use two different approaches. We obtain two 
distinct sequences of continuous functions  $\Phi_1$ on $X$ and $\Phi_2$ on $X$ and
invariant Gibbs measures $\mu_1$ associated to $\Phi_1$ and  $\mu_2$ associated to $\Phi_2$, 
such that $\pi\mu_1=\pi\mu_2=\nu$ (see Proposition \ref{key2}). In general, 
we have $\mu_1\neq \mu_2$. Therefore, in Proposition \ref{generalp}, we study a condition for $\mu_1= \mu_2$.
To finish the section, we examine a condition under which a unique invariant Gibbs measure $\mu$ for 
$\phi \in Bow (X)$ is projected to a unique
invariant Gibbs measure $\nu$ for a function that belongs to the Bowen class. 
If $\phi=\psi\circ\pi$ where $\psi\in Bow(Y)$, then we have the result. 
In Proposition \ref{small}, we extend this $\phi$ slightly to a more general function.

Finally, in Section \ref{sectioncomp}, we relate our results to the existing theory of factors of Gibbs equilibrium states for continuous functions.
Theorem \ref{main1} shows that, given a unique invariant Gibbs measure $\mu$ for an almost additive potential 
$\F=\{\log f_n\}_{n=1}^{\infty}$ on $X$ with bounded variation, $\pi\mu$ is a unique invariant Gibbs measure for 
an asymptotically subadditive potential $\G=\{\log g_n\}_{n=1}^{\infty}$ 
on $Y$ with bounded variation. 
Pollicott and Kempton \cite{PK} considered this question in a continuous case. Given a unique invariant Gibbs measure $\mu$ for a 
function $f$ on $X$ of summable variation, where $X$ is a full shift, they found a continuous function $g$ on $Y$ such that $\pi\mu$ is a Gibbs measure for $g$. 
We show in Proposition \ref{svcase} that our $\G=\{\log g_n\}_{n=1}^{\infty}$ is a generalization of the continuous function $g$ 
obtained in \cite{PK}. 

After this paper was accepted following minor revisions, I found that the related results have been studied 
by Barral and Feng \cite{BF} in their work on the weighted thermodynamic formalism. 
There, in Theorem 3.1 (i),
they showed a variational principle concerning an equilibrium state $\mu$
for a sequence of continuous functions and the image measure under a
factor map in a general setting. This is related to our Theorem 3.1.
The results of Theorem 3.1 (i) in \cite{BF} are obtained by using
relative pressure. In contrast, our Theorem 3.1 is proved by only using
the properties of
Gibbs measures. We obtain the appropriate potential for the image of a
Gibbs measure by observing the properties of a factor of the Gibbs
measure. Also, Theorem 3.1 (iii) \cite{BF} states the relationship between
$\mu$ and $\pi\mu$, and is related to our Theorem 4.8. However, while
Theorem 3.1 (iii) is studied in a more general setting, we are able to
obtain a more detailed result in Theorem 4.8.

\section{Background}\label{background}
 
We first summarize the basic definitions in symbolic dynamics. For notation and terminology not explained 
here, see \cite{LM}.  
$(X, \sigma_X)$ is a {\em one-sided subshift} if $X$ is a closed
shift-invariant subset of $\{1,\cdots, k\}^{\N}$ for some $k\geq
1$, i.e., $\sigma_X(X)\subseteq X$,  where the shift $\sigma_X:X\rightarrow X$ is defined by
$\sigma_{X}(x)=x', \text { for } x=(x_n)^{\infty}_{n=1}, x'=
(x'_n)^{\infty}_{n=1}\in X, x'_{n}=x_{n+1}\text{ for all } n\in \N$. Define a metric $d$  on $X$ by 
\begin{equation*}
d(x, x') = \left\{
\begin{array}{rl}
{1}/{2^{k}} & \text{if } x_i={x'}_i \text{ for all }1\leq i\leq k \text{ and } x_{k+1}\neq {x'}_{k+1},\\
1 & \text{if } x_1\neq x'_1,\\
0 & \text{otherwise. }
\end{array} \right.
\end{equation*}

For each $n \in \N,$ denote by
$B_n(X)$ the set of all $n$-blocks that occur in points in $X$.
$x_1 \dots x_n$ is an allowable word of length $n$ if $x_1\dots
x_n\in B_n(X)$. 
If $x_1\cdots x_n \in B_n(X),$ then $[x_1 \cdots
x_n]$ is a cylinder in $X$. A subshift $(X,\sigma_X)$ is {\em irreducible} if for any allowable words $u, v$ of $X$, 
there exists an allowable word $w$ of $X$ such that $uwv$ is allowable. If, in addition, 
there always exists a word $w$ with a fixed length $p>0$ such that $uwv$ is allowable, then  
$(X, \sigma_X)$ has the {\em specification property}. 

Let $(X, \sigma_X)$ and $(Y, \sigma_Y)$ be subshifts and $\pi:X\rightarrow Y$ be a factor map. 
If the $i$-th position of the image of
$x$ under $\pi$ depends only on $x_i,$ then $\pi$ is a one-block
factor map. A shift of finite type $(X,\sigma_X)$ is {\em one-step} if there exists a set $F$ of forbidden blocks of length $\leq 2$ such 
that $X=\{x\in \{1, \cdots, k\}^{\N}: \omega \textnormal{ does not appear in } x \textnormal{ for any }\omega\in F\}$. 
A subshift is called a {\em sofic shift} if it is the image of a shift of finite type under a factor map. 
It is known that an irreducible sofic shift is the image of a one-step irreducible shift of 
finite type under a one-block factor map (see \cite{LM}). \label{En}
Throughout the paper, we assume that $\pi$ is a
one-block factor map and any shift of finite type $(X, \sigma_X)$ is one-step. Denote by $M(X, \sigma_X)$ the collection of all $\sigma_X$-invariant
Borel probability measures on $X$
and by $Erg(X, \sigma_X)$ all ergodic members of $M(X, \sigma_X)$. 

Next we give a brief overview of the results in pressure theory for almost additive potentials and subadditive potentials.
These generalize the work of Ruelle and Walters on theory of pressure for continuous functions.

As an application, pressure theory for sequences of continuous functions has been used to study dimension problems on  
non-conformal expanding maps. Let $T$ be the endomorphism of the torus given by $T(x,y)=(lx
\textnormal{ mod }1, my \textnormal{ mod }1 ), l>m\geq 2, l,m \in
\N$. The Hausdorff dimension of compact invariant subsets of $T$ has been widely studied \cite{B, Mc, KP, KP2, Y1, EO}.
In \cite{Feng, Y2, Y3}, this problem  was studied by using equilibrium states for sequences of continuous 
functions.
The factor of equilibrium states for sequences of continuous functions and 
their preimages were studied (see Section 4 and Example 5.2 \cite{Y2}). 


Let $(X, \sigma_X)$ be a subshift. For each $n\in \N$, let $\phi_n:X\rightarrow \R^{+}$ be a continuous function. 
Define a sequence of continuous functions $\Phi=\{\log \phi_n\}^{\infty}_{n=1}$ and suppose that $\Phi$
satisfies the subadditivity condition, i.e., for every $n,m\in \N$
and $x \in X$, $\phi_{n+m}(x)\leq \phi_{n}(x)
\phi_{m}(\sigma_X^{n}(x))$. We note that this is equivalent to $\{\log \phi_n\}^{\infty}_{n=1}$ being subadditive.
Then $\Phi$ is a {\em subadditive
potential} on $(X, \sigma_X)$. $\Phi$ is an {\em almost
additive potential} if there is a constant $C>0$ such that for every $n,m\in \N$ and for every $x\in
X$, $$e^{-C}\phi_{n}(x)\phi_{m}(\sigma_X^{n} x)\leq \phi_{n+m}(x)\leq e^C
\phi_{n}(x)\phi_{m}(\sigma_X^{n} x). \label{defaa}$$
Define $$M_{n}= \sup\{\frac{\phi_n(x)}{\phi_n(x')}: x,x'\in X, x_i=x'_i, \text { for all } 1 \leq i\leq n\}.$$ 
Then $\Phi$ has {\em bounded variation} if there exists a constant $M>0$ such that $\sup_{n\in \N}M_n\leq M$. 

The topological pressure $P_X(\Phi)$ of an almost additive potential $\Phi$ on $X$ 
was defined by Barreira \cite{B2006} and Mummert \cite{m2006} and the variational principle was also shown. 

If $f:X \rightarrow \R$ is a continuous function, 
we let $\phi_n(x)=e^{f(x)+\dots+f(\sigma^{n-1}_Xx)}, n\in \N,$ and 
define $\Phi=\{\log \phi_n\}_{n=1}^{\infty}$. Then $\Phi$ is an additive sequence. Also, we note that   
$\lim_{n \rightarrow \infty}{1}/{n} \int \log \phi_n \  d\mu= \int f \  d\mu$  for any $\mu\in M(X, \sigma_X)$. 
If $\Phi$ is a subadditive potential, this also holds by the subadditive ergodic theorem 
(see Theorem~10.1 in~\cite{W1}). 
For $f\in C(X)$, define 
$$V_n(f)=\sup\{f(x)-f(x'): x, x'\in X,  
x_i=x'_{i}, 1\leq i\leq n\}$$ 
and $\sigma_nf=\sum_{i=0}^{n-1}f\circ \sigma^{i}.$ 
Then the Bowen class $Bow(X)$ is the set 
$Bow (X)=\{f\in C(X): 
\sup_{n\in \N}V_{n}(\sigma_nf)<\infty\}$
(see \cite{W3}).
If $f\in Bow (X)$, then it is known from \cite{W3} that the equilibrium state for $f$ is unique and it is the unique invariant 
Gibbs measure for $f$. 
In particular, $Bow(X)$ contains the functions of summable variation (see \cite{W3}).     
We note that if $f\in Bow(X)$, then the additive potential $\Phi$
has bounded variation.  

Theorems \ref{BM} below generalize   
the variational principle for continuous functions to that for sequences of continuous functions and 
Theorem \ref{gibbs0} extends the theory of equilibrium states for continuous 
functions to almost additive potentials. In Sections \ref{sectionfactor}, \ref{sectionrelativepressure} and \ref{sectionpreimage} we will study Gibbs measures for 
sequences of continuous functions and the following 
results play important roles.

\begin{theorem}\label{BM}\cite{B2006,m2006}
Let $(X,\sigma_X)$ be a subshift
and $\Phi=\{\log \phi_n\}^{\infty}_{n=1}$ be an almost additive
potential on $(X, \sigma_X)$ with bounded variation. Then
\begin{equation}\label{vp}
P_X(\Phi)=\sup_{\mu\in
M(X,\sigma_X)}\{h_\mu(\sigma_X)+\lim_{n\rightarrow\infty}\frac{1}{n}\int\log \phi_n
d\mu\}=\sup_{\mu\in
M(X,\sigma_X)}\{h_\mu(\sigma_X)+\int \lim_{n\rightarrow\infty}\frac{1}{n}\log \phi_n
d\mu\}
\end{equation}
where the topological pressure $P_X(\Phi)$ is defined by 
$$P_X(\Phi)=\lim_{n \rightarrow
\infty}\frac{1}{n}\log (\sum_{i_1\cdots i_n \in B_n(X)}e^{\sup \log {\phi_n(x)}}),$$ 
and where
the supremum is taken over  $x\in [i_1\cdots i_n], n \in \N$. 
\end{theorem}
\begin{remark}
The conditions that the subshift be of finite type and be topologically mixing \cite{B2006,m2006}
are not necessary, the same result holds for 
general subshifts (see Theorem 1.1 \cite{CFH}).
\end{remark}

On the other hand, the topological pressure $P_X(\Phi)$ of a subadditive potential $\Phi=\{\log \phi_n\}_{n=1}^{\infty}$ on a subshift 
$(X, \sigma_X)$ was studied 
by Cao, Feng and Huang \cite{CFH} and a variational principle was also shown.  Let $n\in \N$ and $\epsilon>0$. 
A subset $E$ of $X$ is an {\em $(n, \epsilon)$ separated subset} of 
$X$ if $\max_{0\leq i\leq n-1}d(\sigma_X^{i}x, \sigma_X^{i}y)>\epsilon$ for all $x, y \in E, x\neq y$.
For a subadditive potential $\Phi$ on $X$, define
$$P_n(\Phi, \epsilon)=\sup\{\sum_{x\in E}\phi_n(x):E \text { is an } (n, \epsilon) \text{ separated subset of } X\},$$
and let $P(\Phi, \epsilon)=\limsup_{n\rightarrow \infty}({1}/{n})\log P_{n}(\Phi,\epsilon).$ The topological pressure for   
a subadditive potential $\Phi$ is defined by $P_X(\Phi)=\lim_{\epsilon\rightarrow 0}P(\Phi, \epsilon).$  
Then (\ref{vp}) in Theorem \ref{BM} holds if an almost additive potential is replaced by a subadditive potential \cite{CFH}.  





Feng and Huang \cite{FH} also considered sequences of continuous functions called asymptotically subadditive potentials, which are a 
generalization of subadditive potentials and almost additive potentials. 
A sequence of continuous functions $\Phi=\{\log \phi_n\}_{n=1}^{\infty}$ on a subshift $(X, \sigma_X)$ is 
an {\em asymptotically subadditive} if
for any $\epsilon>0$ there exists a subadditive potential $\Psi=\{\log \psi_n\}_{n=1}^{\infty}$ on $X$ such that 
$\limsup_{n\rightarrow\infty} (1/n) \sup_{x\in X} \vert \log \phi_n(x)-\log \psi_n(x)\vert \leq \epsilon$. \label{defas} 

The topological pressure of an  asymptotically subadditive potential is defined in the same manner as it is defined 
for a subadditive potential.
Then (\ref{vp}) in Theorem \ref{BM} is valid when we replace an almost additive potential  
by an asymptotically subadditive potential \cite{FH}. We note that an almost additive potential 
$\Phi=\{\log \phi_n\}_{n=1}^{\infty}$ on $X$ satisfying 
$e^{-C}\phi_n(x)\phi_m(\sigma^n_Xx)\leq \phi_{n+m}(x)\leq e^{C}\phi_n(x)\phi_m(\sigma^n_Xx)$
is an asymptotically subadditive potential on $X$ by
setting a subadditive potential   
$\Psi=\{\log \phi_ne^{C}\}_{n=1}^{\infty}$ in the above equation. 

\begin{definition}\label{defeq}
Let $\Phi=\{\log \phi_n\}_{n=1}^{\infty}$ be a sequence of continuous functions on $X$. 
A $\bar \mu \in M(X, \sigma_X)$ is an {\em equilibrium state} for 
$\Phi$ if 
$$h_{\bar \mu}(\sigma_X)+ \lim_{n \rightarrow
\infty}\frac{1}{n}\int \log \phi_n d\bar\mu =\sup_{\mu\in M(X, \sigma_X)}\{h_{\mu}(\sigma_X)+\lim_{n \rightarrow
\infty}\frac{1}{n}\int\log \phi_n d\mu\}.$$
\end{definition}
Similarly, $\bar \mu\in M(X, \sigma_X)$ is an equilibrium state for 
a Borel measurable function $f$ on $X$  if 
$$h_{\bar \mu}(\sigma_X)+ \int f d\bar\mu 
=\sup_{\mu\in M(X, \sigma_X)}\{h_{\mu}(\sigma_X)+\int f d\mu\}.$$
We denote by $M_{\Phi}(X, \sigma_X)$ the set of 
equilibrium states for $\Phi$.

The definition of a Gibbs measure can be extended to a sequence of continuous functions. 

\begin{definition}\label{defgb}
Let $\Phi=\{\log \phi_n\}^{\infty}_{n=1}$ be an asymptotically subadditive  
potential on a subshift $(X,\sigma_X)$. A Borel probability measure $\mu$ on $X$ is
a {\em Gibbs measure} for $\Phi$ if there exists $C_0>0$ such
that
\begin{equation}\label{gp}
\frac{1}{C_0}<\frac{\mu([x_1 x_2 \cdots x_n])}{e^{-nP_X(\Phi)}\phi_n(x)}<C_0
\end{equation}
for every $x \in X$ and $n\in \N$.
\end{definition}

The next theorem is used throughout the paper.
\begin{theorem}\cite{B2006, m2006}\label{gibbs0}
Let $(X, \sigma_X)$ be a subshift with the specification property. 
Then there exists a unique invariant Gibbs measure for an almost additive potential $\Phi$ on $X$ with bounded variation and it is the
unique equilibrium state for $\Phi$. It is also mixing.
\end{theorem}

Let $(X, \sigma_X)$ and $(Y, \sigma_Y)$ be subshifts and  
let $\pi:X\rightarrow Y$ be a factor map between subshifts.
The main goal of this paper is to study $\pi\mu$ when 
$\mu\in M(X, \sigma_X)$  
is a unique invariant Gibbs measure for an almost additive 
potential. For this purpose, in Sections \ref{sectionrelativepressure} and  
\ref{sectionpreimage}, we will use relative pressure theory. 
Relative pressure for continuous functions and the relative variational principle (see \cite{LW, W2}) were extended to 
subadditive potentials under a certain condition (see \cite{ZC, Y3}). Here we  state 
some basic results from \cite{ZC, Y3} that we need.  \\

Let $(X, \sigma_X), (Y, \sigma_Y)$ be subshifts and  
$\pi:X\rightarrow Y$ be a factor map.\label{defirp}
Let $\Phi=\{\log \phi_n\}_{n=1}^{\infty}$ be a subadditive potential on $(X, \sigma_X)$. 
Let $n\in \N$ and $\epsilon>0$.
For each $y\in Y$, define
$$P_n(\sigma_X, \pi, \Phi, \epsilon)(y)=\sup\{\sum_{x\in E}\phi_n(x):E \text { is an } (n, \epsilon) \text{ separated subset of } \pi^{-1}\{y\}\},$$
$$P(\sigma_X, \pi, \Phi, \epsilon)(y)=\limsup_{n\rightarrow \infty}\frac{1}{n}\log P_{n}(\sigma_X, \pi, \Phi,\epsilon)(y),$$
$$P(\sigma_X, \pi, \Phi)(y)=\lim_{\epsilon\rightarrow 0}P(\sigma_X, \pi,\Phi, \epsilon)(y).$$
The definitions above are a generalization of the usual definitions for the relative pressure 
for continuous functions (see \cite{W2}) and $P(\sigma_X, \pi, \Phi):Y\rightarrow [-\infty, \infty)$ is Borel measurable.
 
Suppose that $\Phi$ has bounded variation. 
In addition, suppose there exists $C>0$ such that $e^{-C}\phi_n(x)\phi_m(\sigma^{n}_X x)\leq 
\phi_{n+m}(x)$ or
$\phi_n(x)$ depends on the first $n$ coordinates of
$x\in X$. 
Then by the proof of Theorem~3.4 in~\cite{Y3}, we have 
\begin{equation}\label{nicef}
P(\sigma_X, \pi, \Phi)(y)=\limsup_{n\rightarrow\infty}\frac{1}{n}\log 
\big (\sum_{x\in D_n(y)}\phi_n(x)\big),
\end{equation}
where $D_n(y)$ is a set consisting of one point from each nonempty set  
$\pi^{-1}(y)\cap [x_1\dots x_n]$ in $X$.

\begin{theorem}\cite{ZC}(A special case of the relative variational principle)\label{rvpforsub}\\
Let $(X,\sigma_X), (Y, \sigma_Y)$ be subshifts, $ \pi:X\rightarrow Y$ a factor map and
$\Phi=\{\log \phi_n\}_{n=1}^{\infty}$ a subadditive potential on $X$.
Then for each $m\in M(Y, \sigma_Y)$ such that there exists  $\mu\in M(X, \sigma_X)$ with $\pi\mu=m$ and $\lim_{n\rightarrow \infty}(1/n)\int\log \phi_{n}d\mu\neq -\infty$,
\begin{equation*}
\int_{Y} P(\sigma_X, \pi, \Phi)dm=\sup\{h_{\mu}(\sigma_X)- h_{m}(\sigma_Y)+\lim_{n\rightarrow \infty}\frac{1}{n}\int_{X} \log \phi_n d\mu : \mu\in M(X, \sigma_X) \text{ and }
\pi\mu=m\}.
\end{equation*}
\end{theorem}

Under the assumptions of Theorem \ref{rvpforsub}, 
$\bar \mu\in M(X, \sigma_X)$ is a {\em relative equilibrium state } for $\Phi$ over $m$ if $\pi \bar\mu=m$ and
\begin{align*} 
&h_{\bar \mu}(\sigma_X)-h_{m}(\sigma_Y)+\lim_{n\rightarrow \infty}\frac{1}{n}
\int \log \phi_n d\bar\mu\\&=\sup\{h_{\mu}(\sigma_X)-h_{m}(\sigma_Y)+\lim_{n\rightarrow \infty} 
\frac{1}{n}\int \log \phi_n d\mu  :\mu\in M(X,\sigma_X) \textnormal{ and }\pi\mu=m\}
\end{align*} 
(see \cite{W2}).
\label{desfre}\\

We will study the relation between $\mu$ and and $\pi\mu$ in Theorems~\ref{thm1} and~\ref{thm2} in Section~\ref{sectionpreimage}.

\section{Factors of generalized Gibbs measures for almost additive potentials}\label{sectionfactor}
In this section, we study the image of a unique invariant Gibbs measure for 
an almost additive potential with bounded variation under a factor map. 
We characterize it as a unique invariant Gibbs measure
for an asymptotically subadditive potential with bounded variation.  

Let $(X, \sigma_X)$ and $(Y, \sigma_Y)$ be subshifts and $\pi:X\rightarrow Y$ be a factor map.
Throughout this section, we shall take 
$\F=\{\log f_n\}_{n=1}^{\infty}$ to be an almost additive potential on $X$. Let $C>0$ be a constant   
such that  
\begin{equation}\label{aaconst}
e^{-C}f_n(x)f_m(\sigma^n_Xx)\leq f_{n+m}(x)\leq  e^{C}f_n(x)f_m(\sigma^n_Xx), \text{ for all } x\in X, m, n \in \N.
\end{equation}
In addition, we assume that $\F$ has bounded variation and let $M>0$ be a constant such that  
\begin{equation}\label{bvforf}
\sup_{n\in\N}\{\frac{f_n(x)}{f_n(x')}: x_i=x'_i \text { for all } 1\leq i \leq n\}\leq M. 
\end{equation}
Such a constant exists because $\F$ has bounded variation.
For all $n\in \N, y=(y_1,\dots, y_n,\dots)\in Y$, denote by $E_n(y)$ a set consisting of exactly one point from each cylinder  
$[x_1\dots x_n]$  such that  $\pi([x_1\dots x_n])\subseteq [y_1\dots y_n]$. Define
\begin{equation*}
g_n(y)= \sup_{E_n(y)}\{\sum_{x \in E_n(y)} f_n(x)\}.
\end{equation*}
We note that if $(X, \sigma_X)$ is irreducible, then $E_n(y)$ is a set consisting of exactly one point from 
each cylinder $[x_1\dots x_n]$ such that $\pi(x_1\dots x_n)=y_1\dots y_n$. 
Define  $\tilde{g}_n(y)=g_n(y)e^{-nP_X(\F)}$.
Let $\G=\{\log g_n\}_{n=1}^{\infty}$ and 
$\widetilde \G=\{\log \tilde{g}_n \}_{n=1}^{\infty}$.
Also define a sequence of continuous functions $\H=\{\log \tilde {g}_ne^{C}\}_{n=1}^{\infty}$ on $Y$. 
We shall continue to use this notation throughout the rest of this section.

We recall from Theorem \ref{gibbs0} that if $\F$ is an almost additive potential with bounded variation, 
then there is a unique invariant Gibbs measure for $\F$ and it is the unique equilibrium state for $\F$. 

The main goal of this section is to prove the following theorem. 

\begin{theorem}\label{main1}
Let $(X,\sigma_X)$, $(Y,\sigma_Y )$ be subshifts
and $\pi:X \rightarrow Y$ a factor map.  Suppose that $X$ has the specification property.
Let $\F=\{\log f_n\}_{n=1}^{\infty}$ be an almost additive potential on $X$ 
with bounded variation. Let $\mu\in M(X, \sigma_X)$ be the unique Gibbs measure for $\F$ and $\nu=\pi\mu\in M(Y, \sigma_Y)$.
Then $\nu$ is the unique invariant Gibbs measure for the asymptotically subadditive potential $\G=\{\log {g}_n\}_{n=1}
^{\infty}$ on $Y$ with bounded variation. 
It is the unique equilibrium state for $\G$ and it is mixing.  Then 
\begin{align}  
P_X(\F)&=\sup\{h_{\bar\mu}(\sigma_X)+\lim_{n\rightarrow\infty}\frac{1}{n}\int 
\log f_n d\bar\mu: \bar\mu\in M(X, \sigma_X)\} \label{eqimp1} \\ 
&=\sup\{h_{\bar\nu}(\sigma_Y)+\lim_{n\rightarrow\infty}\frac{1}{n}\int \log g_n d\bar\nu: 
\bar \nu\in M(Y, \sigma_Y)\}\\ 
&=P_Y(\G). \label{eqimp2}
\end{align}
\end{theorem}

In the special case that $\mu$ is a Gibbs measure associated to a single potential 
$f\in Bow(X)$, rather than being an almost additive potential, 
we obtain the following corollary.  

\begin{coro}\label{maincoro}
Let $(X,\sigma_X)$, $(Y,\sigma_Y )$ be subshifts and  
$\pi:X \rightarrow Y$ a factor map. Suppose that $X$ has the specification property.
Let $\mu\in M(X, \sigma_X)$ be the unique invariant Gibbs measure for $f\in Bow(X)$ and $\nu=\pi\mu\in M(Y, \sigma_Y)$.
Define $$\bar g_n(y)= \sup_{E_n(y)} \{\sum_{x \in E_n(y)}e^{f(x)+\dots+f(\sigma^{n-1}_Xx)}\}.$$
Then $\nu$ is the unique invariant Gibbs measure for the subadditive potential   
$\bar \G=\{\log {\bar g}_n\}_{n=1}^{\infty}$ on $Y$ with bounded variation. It is the unique equilibrium state for $\bar \G$ and it is mixing. Then 
\begin{align}\label{eqimp} 
P_X(f)&=\sup\{h_{\bar \mu}(\sigma_X)+\int f  d\bar \mu: \bar\mu\in M(X, \sigma_X)\}\\
&=\sup\{h_{\bar\nu}(\sigma_Y)+\lim_{n\rightarrow\infty}\frac{1}{n}\int \log \bar g_n d\bar\nu: 
\bar \nu\in M(Y, \sigma_Y)\}\\
&=P_Y(\bar \G).
\end{align}
\end{coro}
\begin{remark}\label{re1}
The potential $\G=\{\log g_n\}_{n=1}^{\infty}$ here is slightly different from the potential 
$\Psi=\{\log \psi_n\}_{n=1}^{\infty}$ found in Theorem 3.1 of \cite{BF}. $\Psi$ in \cite{BF} 
was studied by using the theory of relative pressure in a more general setting than that pursued 
here. The definition of $\psi_n (y)$ involves the entire sequence $y$ and, in general, $\psi_n$ is not a locally constant function. Our approach
has been to find $\G$ using the properties of the image of a Gibbs
measure and $g_n$ is a locally constant function. 
It seems that equalities (\ref{eqimp1})--(\ref{eqimp2}) could be obtained from Theorem 3.1 (i) \cite{BF} by replacing $\Psi$ by $\G$, 
but due to the difference in techniques required to prove the result, this does not work in general. However, we note that 
Theorem 3.5 \cite{BF}
deals with the special case in which $\pi$ is a factor map between full shifts: Barral and Feng use our $\G$ to study the image $\pi\mu(I)$ 
for a cylinder set $I$ of length $n$.  
\end{remark}

In statistical mechanics, non-Gibbsian measures have been often found to occur 
as images of Gibbs measures under Renormalization Group transformations.
The question of when this phenomenon occurs has been widely studied and possible generalizations of Gibbs 
measures have also been considered.
For example, see \cite{VFS, MMR, VE}.
Since the projections and Renormalization Group maps share some mathematical properties, the above theorem and 
corollary may be applicable to these areas.
Studying the continuity of 
the function $F_1(y)=\lim_{n\rightarrow\infty}(\log {g}_n)/n$ in Theorem \ref{main1}  and 
$F_2(y)=\lim_{n\rightarrow\infty}(\log {\bar g}_n)/n$ in Corollary \ref{maincoro}, will tell us
when the projection $\nu$ is an invariant (possibly weak) Gibbs measure for a continuous function.  

We stress that the image of the invariant Gibbs measure for a continuous function need not be an invariant 
Gibbs measure for a continuous function but may be for a sequence of continuous functions. 
In Example 4.1 \cite{K}, the image of the Gibbs measure for a H\"older  continuous 
function is not an invariant Gibbs measure for a continuous function defined in \cite{K}. However, by Corollary \ref{maincoro},
it is a unique invariant Gibbs measure for a subadditive potential. 
  
In order to prove Theorem \ref{main1}, we start with the following lemma.
In the following lemmas, propositions and theorems, we continue to use $\F, \G, \widetilde\G, \H, E_n(y)$ defined at the beginning
of this section.

\begin{lemma}\label{subd}
Let $(X,\sigma_X)$, $(Y,\sigma_Y )$ be subshifts and 
$\pi:X \rightarrow Y$ be a factor map. 
Then $\H=\{\log \tilde {g}_ne^{C}\}_{n=1}^{\infty}$ is a subadditive potential on $Y$ and so
 $\widetilde \G=\{\log \tilde{g}_n\}_{n=1}^{\infty}$ is an asymptotically
subadditive potential on $Y$.
\end{lemma} 
\begin{proof}
We first show that, for $n, m\in \N, y\in Y$, $\tilde{g}_{n+m}(y)\leq \tilde{g}_n(y)\tilde{g}_m(\sigma^{n}_Y y)e^{C}$. 
Let $y\in Y$ and take a set $E_{n+m}(y)$.
Let $x=(x_1, \dots,  x_{n},x_{n+1},\dots, x_{n+m}, \dots)\in E_{n+m}(y)$. 
Noting that we can construct a set $E_n(y)$ such that $x\in E_n(y)$ and a set $E_{m}(\sigma^{n}_Yy)$ 
such that $\sigma^n_Xx\in E_{m}(\sigma^{n}_Yy)$, we obtain 
\begin{align*}
 &\sum_{x\in E_{n+m}(y)} f_{n+m}(x)e^{-(n+m)P_X(\F)}
\leq \sum_{x\in E_{n+m}(y)} f_{n}(x)f_m(\sigma^{n}_Xx)e^{-(n+m)P_X(\F)} e^{C}\\
&\leq \tilde g_n(y)\tilde g_n(\sigma^{n}_Yy)e^C.
\end{align*}
Taking the supremum over $x\in E_{n+m}(y)$, we obtain 
\begin{equation}\label{nearsub}
 \tilde {g}_{n+m}(y)\leq \tilde {g}_{n}(y) \tilde {g}_{m}(\sigma^{n}_Yy)e^{C}.
\end{equation}
Now for each $n\in \N, y\in Y$, let  $h_n(y)=e^{C}\tilde g_{n}(y)$. Then, (\ref{nearsub}) implies that 
$h_{n+m}(y)\leq h_n(y)h_m(\sigma^{n}_Yy)$ and so $\H$ is a subadditive potential on $Y$. From the definition, it is easy to see that 
$\widetilde \G$ is an asymptotically subadditive potential on $Y$. 
\end{proof}

\begin{proposition} \label{gibbs1}
Let $(X,\sigma_X)$, $(Y,\sigma_Y )$ be subshifts and $\pi:X \rightarrow Y$  a factor map. Suppose that $X$ has 
the specification property. 
Suppose that $\F=\{\log f_n\}_{n=1}^{\infty}$ is an almost additive potential on $X$ with bounded variation.
Then there exists a unique invariant Gibbs measure $\nu_{\widetilde \G}$ for 
$\widetilde\G=\{\log \tilde {g}_n\}_{n=1}^{\infty}$ on 
$Y$. It is the unique equilibrium state for $\widetilde \G$ and it is mixing.
\end{proposition}
We postpone the proof of Proposition \ref{gibbs1} until the end of this section.

\begin{proposition}\label{imp}
For $\widetilde \G= \{\log \tilde{g}_n \}_{n=1}^{\infty}$ on $Y$ in Proposition \ref{gibbs1}, we have 
$P_{Y}(\widetilde \G)=0$.
\end{proposition}
\begin{proof}
We first note that $\G=\{\log g_n\}_{n=1}^{\infty}$ and $\tilde \G=\{\log \tilde{g}_n\}_{n=1}^{\infty}$ have bounded variation 
because $g_n$ is a locally constant function 
that depends on the first
$n$ coordinates of $y\in Y$.
Since $\G$ and $\tilde \G$  
are asymptotically subadditive potentials on $Y$, 
we have from the variational principle for asymptotically subadditive potentials that
\begin{align}\label{kindofvp}
 P_{Y}(\widetilde \G)&=
\sup \{h_{m}(\sigma_Y)+\lim_{n\rightarrow \infty}\frac{1}{n}\int \log g_n(y)e^{-nP_{X}(\F)}dm: m\in M(Y, \sigma_Y)\}\\
&=\sup\{h_{m}(\sigma_Y)+\lim_{n\rightarrow \infty}\frac{1}{n}\int \log g_n(y)dm: m\in M(Y, \sigma_Y)\}-P_{X}(\F)\\
&=P_{Y}(\G)-P_{X}(\F).
\end{align}
Using the fact that $\G$ has 
bounded variation, the definition of topological pressure for asymptotically subadditive potentials gives us
\begin{equation}\label{topo1}
P_{Y}(\G)=\limsup_{n\rightarrow \infty}\frac{1}{n}
\log (\sum_{y_1\dots y_n\in B_n(Y)}g_n(y)),
\end{equation}
where $y$ is any point from the cylinder $[y_1\dots y_n]$. 
Let $N_{n}=\sum _{y_1\dots y_n\in B_{n}(Y)}g_n(y)$, where $y$ is any point from the cylinder $[y_1\dots y_n]$.
Since $\F$ is almost additive with bounded variation, by Theorem \ref{BM}, we have 
\begin{equation}\label{topo2}
P_{X}(\F)=\lim_{n\rightarrow \infty}\frac{1}{n}\log 
(\sum_{z_1\dots z_n\in B_{n}(X)}e^{\sup \log f_n(z)}),
\end{equation}
where the supremum is taken over all 
$z\in [z_1\dots z_n], n\in \N$. 
Let $G_n=\sum_{z_1\dots z_n\in B_{n}(X)}e^{\sup \log f_n(z)}$, where the supremum is taken over all 
$z\in [z_1\dots z_n]$.

Now we show that $G_n\geq {N_n}/M$. Since $\F$ satisfies (\ref{bvforf}), for each $z_1\dots z_n\in B_n(X)$, 
\begin{equation*}
 \frac{1}{M}\sup\{f_n(z):z\in [z_1\dots z_n]\}\leq f_n(z) \text { for any } z\in [z_1\dots z_n]. 
\end{equation*}
Thus we obtain, 
\begin{align*}
 e^{\sup_{z\in [z_1\dots z_n]}\log f_n(z)}&\geq e^{ 
\log\frac{\sup\{f_n(z):z\in [z_1\dots z_n]\}}{M}} \\
&=\frac{1}{M}\sup\{f_n(z):z\in [z_1\dots z_n]\}.
\end{align*}
Let $y_1\dots y_n\in B_n(Y)$ be fixed. Then 
\begin{align*}
\sum_{\substack{ \pi(z_1 \dots z_n)=y_1 \dots y_n\\ z_1 \dots z_n \in B_n(X)}}e^{\sup_{z\in [z_1\dots z_n]}
\log f_n(z)}&\geq \frac{1}{M}
\sum_{\substack{ \pi(z_1 \dots z_n)=y_1 \dots y_n\\ z_1 \dots z_n \in B_n(X)}}\sup \{f_n(z):z\in [z_1\dots z_n]\} \\
&\geq \frac{1}{M}\sum_{\substack{z\in [z_1\dots z_n]\\ \pi(z_1\dots z_n)=y_1\dots y_n}}
f_n(z).
\end{align*}
Therefore, by the definition of $g_n(y)$, 
$$\sum_{\substack{\pi(z_1 \dots z_n)=y_1 \dots y_n\\ z_1 \dots z_n \in B_n(X)} }e^{\sup_{z\in [z_1\dots z_n]}
\log f_n(z)}\geq \frac{1}{M}g_n(y), y\in [y_1\dots y_n]$$
Summing over all possible $y_1\dots y_n\in B_n(Y)$, we obtain
$$G_n=\sum_{y_1\dots y_n\in B_n(Y)}
(\sum_{\substack{ \pi(z_1 \dots z_n)=y_1 \dots y_n\\ z_1 \dots z_n \in B_n(X)}}
e^{\sup_{z\in [z_1\dots z_n]} \log f_n(z)})\geq \frac{1}{M}N_n.$$

Next we show that $G_n\leq MN_n$. For a fixed $z_1\dots z_n\in B_n(X)$, let $x$ be a fixed point from $[z_1\dots z_n]$.  
For any $z\in [z_1\dots z_n]$, we have $f_n(z)/f_n(x)\leq M$. Therefore,
$\sup\{\log f_n(z):z\in [z_1\dots z_n]\}\leq \log Mf_n(x)$. Using this, 
$e^{\sup_{z\in[z_1\dots z_n]} \log f_n(z)}\leq Mf_n(x)$ for any $x\in [z_1 \dots z_n]$.
Let $y_1\dots y_n\in B_n(Y)$ be fixed. Then 
\begin{equation}\label{eq100}
\sum_{\substack{\pi(z_1 \dots z_n)=y_1 \dots y_n\\ z_1 \dots z_n \in B_n(X)}} e^{\sup \log f_n(z)}
\leq \sum _{\substack{ \pi(z_1 \dots z_n)=y_1 \dots y_n\\ z_1 \dots z_n \in B_n(X)}} Mf_n(x)\leq M g_n(y), y\in[y_1\dots y_n],
\end{equation}
where the supremum in the first summation is taken over $z\in [z_1\dots z_n]$ such that
$\pi(z_1\dots z_n)=y_1\dots y_n$ and $x$ in the second summation is any point from the cylinder  
$[z_1\dots z_n]$ such that\
$\pi(z_1\dots z_n)=y_1\dots y_n.$ Therefore, summing (\ref{eq100}) over all possible $y_1\dots y_n\in B_n(Y)$, 
we obtain $G_n\leq MN_n$.
Hence $(N_n/M)\leq G_n\leq MN_n$. Using (\ref{topo1}) and (\ref{topo2}), we obtain 
$P_{Y}(\G)=P_{X}(\F)$ and this proves the proposition.
\end{proof}

Before studying the potential $\G=\{\log g_n\}_{n=1}^{\infty}$ on $Y$ in Theorem \ref{main1}, 
we first study the potential
$\widetilde \G=\{\log \tilde {g}_n\}_{n=1}^{\infty}$ on $Y$, where  $\log \tilde{g}_n(y)=\log 
g_n(y)e^{-nP_{X}(\F)}, y\in Y$. In the next theorem, we will find that the measure $\nu_{\widetilde \G}$ in Proposition \ref{gibbs1}
is the image of the Gibbs measure $\mu\in M(X, \sigma_X)$ for 
an almost additive potential $\F$ on $X$. 

\begin{theorem}\label{imp1}
Let $(X,\sigma_X)$, $(Y,\sigma_Y )$ be subshifts and $\pi:X \rightarrow Y$ be a factor map. 
Suppose that $X$ has the specification property.
For an almost additive potential 
$\F=\{\log f_n\}_{n=1}^{\infty}$ on $X$ 
with bounded variation, let $\mu\in M(X, \sigma_X)$ be the unique 
equilibrium state which is Gibbs for $\F$ and  $\nu=\pi\mu\in M(Y, \sigma_Y)$. 
Then $\nu$ is the unique invariant Gibbs measure for 
$\widetilde \G=\{\log \tilde g_n\}_{n=1}^{\infty}$ on $Y$ and 
$P_Y(\widetilde \G)=0$. $\nu$ is the unique equilibrium state for $\widetilde \G$. 
\end{theorem}  
\begin{proof}
Since $\mu$ is the unique Gibbs measure for $\F$ on $X$, 
by definition,  there exists $C_1>0$ such that  
\begin{equation}\label{gb1}
\frac{1}{C_1}\leq \frac{\mu([x_1\dots x_n])}{e^{-nP_X(\F)}f_n(x)}\leq C_1, \text { for any } x\in [x_1\dots x_n], n\in \N.
\end{equation}
\end{proof}
It follows from  the definition of $\nu\in M(Y, \sigma_Y)$ that, for each $y_1\dots y_n\in B_n(Y)$,  
\begin{equation}\label{e1}
\nu([y_1\dots y_n])= \mu (\pi^{-1}([y_1\dots y_n]))=\sum_{{\substack{x_1 \dots x_n\in B_n(X)\\
\pi(x_1\dots x_n)=y_1 \dots y_n} }}\mu([x_1 \dots x_n]). 
\end{equation}
By (\ref{gb1}) and (\ref{e1}), 
\begin{align*}
\frac{1}{C_1} \sum_{\substack{x_1\dots x_n \in B_n(X)\\
 \pi(x_1\dots x_n)=y_1\dots y_n} } f_n(x)e^{-nP_X(\F)} & \leq 
\sum_{\substack{x_1\dots x_n \in B_n(X)\\ \pi(x_1\dots x_n)=y_1\dots y_n}}\mu([x_1\dots x_n]) \\
&\leq C_1 \sum_{\substack{x_1\dots x_n \in B_n(X)\\ \pi(x_1\dots x_n)=y_1\dots y_n }} f_n(x)e^{-nP_X(\F)},
\end{align*}
where $x$ in the first and third summations is any point from the cylinder $[x_1 \dots x_n]$ such that 
$\pi(x_1 \dots x_n)=y_1 \dots y_n$.
Therefore,  we obtain  
\begin{equation}\label{gb2}
\frac{1}{C_1}\leq \frac{\nu([y_1\dots y_n])}{\sum_{x_1\dots x_n \in B_n(X), 
\pi(x_1 \dots x_n)=y_1 \dots y_n} f_n(x)e^{-nP_X(\F)}}\leq C_1 
\end{equation}
for any $x\in [x_1\dots x_n], \pi(x_1\dots x_n)=
y_1\dots y_n$. 
Thus, using the property of bounded variation and (\ref{gb2}),  we can find $C_2>0$ such that   
\begin{equation*}
 \frac{1}{C_2}\leq \frac{\nu([y_1\dots y_n])}{\tilde g_n(y)}\leq C_2 
\end{equation*}
for any $n\in \N, y\in [y_1,\dots, y_n]$.
Since $P_Y(\widetilde \G)=0$ (see Proposition \ref{imp}),
\begin{equation}\label{gibbs3}
 \frac{1}{C_2}\leq \frac{\nu([y_1 \dots y_n])}
{e^{-nP_Y(\widetilde \G)}\tilde g_n(y)}\leq C_2 \text{ for any } y\in [y_1\dots y_n].
\end{equation}
Hence $\nu$ is an invariant Gibbs measure for $\widetilde \G$. The rest of the result follows immediately 
from Propositions~\ref{gibbs1} and~\ref{imp}. \\

\noindent \textbf{Proof of Theorem \ref{main1}}\\
Using (\ref{kindofvp}), clearly, $M_{\widetilde \G}(Y, \sigma_Y)=
M_{\G}(Y, \sigma_Y)$. By Theorem \ref{imp1}, $\nu$ is the unique invariant Gibbs measure for 
$\widetilde\G$ satisfying (\ref{gibbs3}) and 
by Proposition \ref{gibbs1} it is the unique equilibrium
state for $\widetilde \G$. Replacing $\tilde g_n(y)$ and $P_{Y}(\widetilde \G)$ in (\ref{gibbs3}) by $g_{n}(y)e^{-nP_X(\F)}$ and 
$P_X(\G)-P_Y(\F)$ respectively,  $\nu$ is the unique 
invariant Gibbs measure for $\G$ and clearly it is the unique equilibrium state for $\G$. The rest follows immediately from (\ref{kindofvp}).\\

\noindent \textbf{Proof of Corollary \ref{maincoro}}\\
Recall from Section \ref{background} that if $f\in Bow(X)$, then $\F=\{\log f_n\}_{n=1}^{\infty}$ 
where $f_n(x)=e^{f(x)+\dots + f(\sigma^n_Xx)}$ is an additive potential with bounded variation.
It is easy to see that $\bar \G$ is a subadditive potential. We apply Theorem \ref{main1} directly to $\F$ and obtain 
the result.\\

We will  use the remainder of this section to study Proposition \ref{gibbs1}. 
In order to prove Proposition \ref{gibbs1}, we make similar arguments to those used in Lemmas 4.6, 4.7 and 4.8 \cite{Y2}. These arguments are based on 
the proofs of Theorems 5 and 7 \cite{B2006}.
Since $X$ has the specification property, for any allowable words $u, v$, there always exists a word $w$ 
with a fixed length $k>0$ such that $uwv$ is allowable. Hence $Y$ also has the specification property with this fixed length $k>0$.     
Since $\tilde g_n$ is a locally constant function, define $a_{i_1\dots i_n}=\tilde{g}_n(y)$, for $y\in [i_1\dots i_n]$.
Define also $S_n=\sum_{i_1\dots i_n\in B_n(Y)}a_{i_1\dots i_n}$ and  $m=\min_{x\in X}f_k(x)$. For all $n\in \N$, let $A_n$ be a set consisting of exactly one point from each cylinder of length $n$ in $Y$. 
Define the Borel probability measure $\nu_n$ on $Y$ concentrated on $A_n$ by 
\begin{equation*}
\nu_{n}=\frac{\sum_{y\in A_n}\tilde g_n(y)\delta _{y}}{\sum_{y\in A_n}\tilde g_n(y)} 
\end{equation*}
where $\delta_{y}$ is the Dirac measure at $y$. Since $\nu_n$ is a Borel probability measure on $Y$ for all $n\in \N$, there exists 
a subsequence $\{\nu_{n_k}\}_{k=1}^{\infty}$ that converges to a Borel probability measure $\nu$ on $Y$ in the weak* topology.
In the following lemmas, 
we continue to use   $k$, $a_{i_1\dots i_n}$, $S_n$ and $m$ as defined above.  
For simplicity, let $E_n(y_1,\dots y_n)$ be a set consisting of exactly one point from each 
cylinder $[x_1\dots x_{n}]$  such that $\pi([x_1\dots x_n])\subseteq 
[y_1\dots y_n]$.

\begin{lemma}\label{step2}
There exist $K_1, K_2>0$ such that $K_1\leq e^{nP_{Y}(\widetilde \G)}/S_n\leq K_2$.   
\end{lemma}
\begin{proof}
Let $l>n$. We show that $S_{l}\leq e^{C}S_nS_{l-n}$. 
Let $i_1\dots i_nj_{1}\dots j_{l-n}\in B_{l}(Y)$.
\begin{align*}
&a_{i_1\dots i_n j_{1}\dots j_{l-n}}\\
&=\sup \{\sum_{x \in E_l(i_1\dots i_n j_1\dots j_{l-n})} 
f_l(x)e^{-lP_{X}(\F)}\}\\
&\leq \sup \{ e^C\sum_{x \in E_l(i_1\dots i_nj_{1}\dots j_{l-n})} 
f_n(x)e^{-nP_{X}(\F)}f_{l-n}(\sigma^n_Xx)e^{-(l-n)P_{X}(\F)}\}\\
&\leq  e^C \sup 
\{\sum_{x \in E_n(i_1\dots i_n)} f_n(x)e^{-nP_{X}(\F)}\} 
\sup \{\sum_{x \in E_{l-n}(j_1\dots j_{l-n})} f_{l-n}(x)e^{-(l-n)P_{X}(\F)}\}\\
&=e^{C}a_{i_1\dots i_n}a_{j_1 \dots j_{l-n}}.
\end{align*}
Therefore, for each fixed $i_1\dots i_n\in B_n(Y)$, 
\begin{equation}\label{in1}
 \sum_{i_1\dots i_nj_{1}\dots j_{l-n}\in B_{l}(Y)}
a_{i_1\dots i_nj_{1}\dots j_{l-n}}\leq e^C
a_{i_1\dots i_n}S_{l-n}.
\end{equation}
Summing over all allowable words $i_1\dots i_n$ of length $n$ such 
that $i_1\dots i_nj_1\dots j_{l-n}$ is allowable, 
we obtain 
\begin{equation}\label{eq001}
S_{l}\leq e^CS_nS_{l-n}.
\end{equation}
Thus $\{\log (e^{C}S_n)\}_{n=1}^{\infty}$ is
subadditive. Since $\widetilde \G$ is asymptotically subadditive (Lemma \ref{subd}), 
by the definition of topological pressure (see \cite{FH}), 
\begin{equation}\label{step1}
P_{Y}(\widetilde \G)=\lim_{n\rightarrow \infty}\frac{\log S_n}{n}=
\lim_{n\rightarrow \infty}\frac{\log e^C S_n}{n}\leq 
\frac{\log e^C S_n}{n}
\end{equation}
for all $n\geq 1$. 
Hence we set $K_2=e^C$. 

Next we show that $S_{l+n}\geq C_0 S_lS_n$ for some $C_0>0$. 
First let $l>n+k$. Since $Y$ is a subshift with 
the specification 
property with a fixed length $k$, for each $i_{1}\dots i_{n}\in B_{n}(Y), j_{1}\dots j_{l-k}\in B_{l-k}(Y)$, there exists
$m_1\dots m_k\in B_{k}(Y)$ such that $i_{1}\dots i_{n}m_1\dots m_k j_{1}$\\
$\dots j_{l-k}\in B_{l+n}(Y)$. 

Fix $i_{1}\dots i_{n}\in B_{n}(Y)$ and $j_{1}\dots j_{l-k}\in B_{l-k}(Y)$. Let $x_1\dots x_n\in B_{n}(X)$ such that $\pi(x_1\dots x_n)=i_1\dots i_n$. Take $x'=(x'_1, \dots x'_{l-k}, \dots)\in X$
such that $\pi(x'_1\dots x'_{l-k})=j_1\dots j_{l-k}$. 
Then there exists a cylinder $\bar{m}_1\dots \bar{m}_k$ of length $k$ in $X$ such that $x_1\dots x_n\bar{m}_1\dots \bar{m}_k$\\
$x'_1
\dots x'_{l-k}\in B_{l+n}(X)$. Now define $\bar x=x_1\dots x_n\bar{m}_1\dots \bar{m}_kx'\in X$. 
We can construct such a $\bar x$
for each given $x_1\dots x_n\in B_{n}(X)$ such that $\pi(x_1\dots x_n)=i_1\dots i_n$ and 
$x'_1\dots x'_{l-k}\in B_{l-k}(X)$ such that $\pi(x'_1\dots x'_{l-k})=j_1\dots j_{l-k}$. 
Below we use the notation $\bar x_{x_1\dots x_n, x'_1\dots x'_{l-k}}$ for $\bar x$ to emphasize that $\bar x$ depends 
on these two allowable words.

For fixed $i_1\dots i_n, j_1\dots j_{l-k}$, we have 
\begin{align*}
&\sum_{i_{1}\dots i_{n}m_1\dots m_k j_{1}\dots i_{l-k}\in B_{n+l}(Y)} 
a_{i_{1}\dots i_{n}m_1\dots m_k j_{1}\dots i_{l-k}\in B_{n+l}(Y)}\\
&\geq e^{-2C}\sum_{i_{1}\dots i_{n} m_1\dots m_k j_{1}\dots i_{l-k}\in B_{l+n}(Y)}
\sup \{\sum_{x\in E_{l+n}(i_1\dots i_n m_1\dots m_k j_{1}\dots j_{l-k})} 
f_n(x) f_k({\sigma}^n x)f_{l-k}({\sigma}^{n+k}x)e^{-(l+n)P_{X}(\F)}\}\\
& \geq e^{-2C} \\
&\times \sum _{\substack {x_1\dots x_{n}\in B_{n}(X)\\ x'_1\dots x'_{l-k}\in B_{l-k}(X) \\ 
\pi(x_1\dots x_{n})=i_1\dots i_{n}
\\ \pi(x'_1\dots x'_{l-k})=j_1\dots j_{l-k}}} f_n(\bar x_{x_1\dots x_n, {x'}_1 \dots {x'}_{l-k}})f_k(
{\sigma}^{n}(
\bar x_{x_1\dots x_n, {x'}_1\dots {x'}_{l-k}}))f_{l-k}({\sigma}^{n+k}(\bar x_{x_1\dots x_n, {x'}_1\dots {x'}_{l-k}}))
e^{-(l+n)P_{X}(\F)}\\
&\geq \frac {e^{-kP_X(\F)-2C}me^{-(n+l-k)P_{X}(\F)}}{M^2} 
\sum_{\substack {x_1\dots x_n\in B_n(X)\\ \pi(x_1\dots x_n)=i_1\dots i_n}}
\sup\{f_n(x):x\in [x_1 \dots x_n]\}  \\
&\times
\sum_{\substack {x_1\dots x_{l-k}\in B_{l-k}(X)\\\pi(x_1\dots x_{l-k})=j_1\dots j_{l-k}}} 
\sup\{f_{l-k}(x):x\in [x_1\dots x_{l-k}]\},
\end{align*}
where $\bar x_{x_1\dots x_n, {x'}_1 \dots {x'}_{l-k}}$ in the second inequality is chosen as explained in the 
preceding paragraph and  
for the last inequality we use the fact that $\F$ has bounded variation. 
Therefore, for each $i_1\dots i_n\in B_{n}(Y), j_1\dots j_{l-k}\in B_{l-k}(Y)$,  we have 
\begin{equation}\label{keypro}
\sum_{i_{1}\dots i_{n} m_1 \dots m_k j_{1}\dots j_{l-k}\in B_{l+n}(Y)} 
a_{i_{1}\dots i_{n}m_1 \dots m_k j_{1}\dots i_{l-k}}\geq 
\frac{e^{-kP_X(\F)-2C}m}{M^2}
a_{i_1\dots i_{n}}a_{j_1\dots j_{l-k}}. 
\end{equation}
Summing over all allowable words $i_{1}\dots i_{n}$ of length $n$ in $Y$ and $j_{1}\dots j_{l-k}$ of length $(l-k)$ in $Y$, 
we obtain
\begin{equation*}
 S_{l+n}\geq \frac{e^{-kP_X(\F)-2C}m}{M^2}S_n S_{l-k}. 
\end{equation*}
Now we claim that $S_l\leq e^{C}M S_k S_{l-k}$. 
For any allowable word $i_{k+1}\dots i_{l}$ in $Y$, there exists
$m_1\dots m_k$ such that  $m_1 \dots m_k i_{k+1}\dots i_{l}$ 
is allowable in $Y$. Let  $i_{k+1}\dots i_{l}$  be fixed.
Then 
\begin{align*}
&\sum_{m_1\dots m_k i_{k+1}\dots i_l\in B_{l}(Y)}
\sup \{\sum_{x\in E_l(m_1\dots m_k i_{k+1}\dots i_{l})} 
f_l(x)\}e^{-lP_X(\F)}\\
&\leq \sum_{m_1\dots m_k i_{k+1}\dots i_l\in B_{l}(Y)}\sup
\{\sum_{x\in E_l(m_1\dots m_k i_{k+1}\dots i_{l})}
f_k(x)f_{l-k}(\sigma^k x)e^C\}e^{-lP_X(\F)}\\
&\leq e^CS_ke^{-(l-k)P_X(\F)}\sum_{\substack{x_1\dots x_{l-k}\in B_{l-k}(X)\\ \pi(x_1\dots x_{l-k})=
i_{k+1}\dots i_l}}\sup\{f_{l-k}(x):x\in [x_1\dots x_{l-k}]\}
\leq e^C M S_k a_{i_{k+1}\dots i_l},
\end{align*}
where for the last inequality we use the fact that $\F$ has bounded variation.
Summing over all allowable words $i_{k+1}\dots i_{l}$, we obtain 
$S_{l}\leq e^{C}MS_{k}S_{l-k}$. 
Hence $S_{l+n}\geq (e^{-3C-kP_X(\F)}m/(M^{3}S_k))S_nS_{l}$ for all $l>k, n\geq 1$. 
For $l+n\leq k+1$, we can also find $C'$ such that $S_{l+n}\geq C'S_lS_n$.
Setting $C''=\min \{C', e^{-3C-kP_X(\F)}m/(M^3S_k)\}$, $\{\log (C''S_n)\}_{n=1}^{\infty}$
is super additive. Therefore, 
\begin{equation*}
P_Y(\widetilde \G)=
\lim_{n \rightarrow\infty}\frac{1}{n}\log S_n=\lim_{n\rightarrow \infty}
\frac{1}{n}\log (C''S_n)\geq \frac{1}{n}\log (C''S_n).  
\end{equation*}
for all $n\geq 1$. Hence we set $K_1=C''$.
\end{proof}

\begin{lemma}\label{step3}
There exist $C_1, C_2>0$ such that
\begin{equation*}
C_1\leq \frac{\nu_{l}([i_1 \dots i_n])}{e^{-n P_Y(\widetilde \G)}\tilde g_n(y)}
\leq C_2 , y\in [i_1\dots i_n] 
\end{equation*}
for all $l, n\in\N, l>n+k$ and cylinders $[i_1\dots i_n]$ in $Y$. Hence $\nu_l$ is a Gibbs measure for $\widetilde \G$. 
\end{lemma}
\begin{proof}
 Let $[i_1 \dots i_n]$ be a fixed cylinder of length $n$ in $Y$. By the definition 
of $\nu_l$, for $n<l$,
\begin{equation*}
 \nu_{l}([i_1\dots i_n])=\frac{\sum_{i_1\dots i_n j_1\dots j_{l-n}\in B_{l}(Y)}
 a_{i_1\dots i_n j_1\dots j_{l-n}}}{S_l}.
\end{equation*}
Given  $j_{k+1}\dots j_{l-n}\in B_{l-n-k}(Y)$, we can find $m_1\dots m_k\in B_{k}(Y)$
such that $i_1\dots i_n m_1\dots m_k j_{k+1}$\\$\dots j_{l-n}\in B_{l}(X)$. Hence
\begin{align*}
&\sum_{i_1\dots i_n j_1\dots j_{l-n}\in B_{l}(Y)}
a_{i_1 \dots i_n j_1\dots j_{l-n}}\\
&\geq e^{-2C}e^{-lP_X(\F)}
\sum_{i_1\dots i_n m_1\dots m_k
j_{k+1}\dots j_{l-n}\in B_{l}(Y)}\sup
\{\sum_{x\in E_n(i_1\dots i_n m_1\dots m_k
j_{k+1}\dots j_{l-n})}f_n(x)f_k({\sigma}^n x)f_{l-(n+k)}({\sigma}^{n+k}x)\}\\
&\geq \frac{e^{-2C-kP_X(\F)}m}{M^2}a_{i_1 \dots i_n}a_{j_{k+1} \dots j_{l-n}}.
\end{align*}
Taking all possible $j_{k+1}\dots j_{l-n}$, we obtain
\begin{equation}
\sum_{i_1\dots i_n j_1\dots j_{l-n}\in B_{l}(Y)}
a_{i_1\dots i_n j_1\dots j_{l-n}}\geq \frac{e^{-2C-kP_X(\F)}m}{M^2}S_{l-n-k}a_{i_1\dots i_n}.
\end {equation}
Hence for $y\in [i_1\dots i_n]$,
\begin{align*}
&\frac{\nu_{l}([i_1 \dots i_n])}{e^{-nP_Y(\widetilde \G)}\tilde g_n(y)}\geq \frac{e^{-2C-kP_X(\F)}m S_{l-n-k}e^{nP_{Y}
(\widetilde \G)}}{M^2 S_l}\\
&\geq \frac{e^{-2C-kP_X(\F)}m e^{nP_Y(\widetilde \G)}}{M^2 e^C S_{n+k}} 
(\text{by } (\ref{eq001}))
\geq \frac{e^{-3C-kP_X(\F)}m K_1}{M^2 e^{kP_{Y}(\widetilde \G)}} 
(\text{by Lemma } \ref{step2}). 
\end{align*}

Similarly,
\begin{equation*}
\frac{\nu_{l}([i_1 \dots i_n])}{e^{-nP_Y(\widetilde \G)}\tilde g_n(y)}
\leq 
\frac{e^C S_{l-n}e^{lP_{Y}(\widetilde \G)}}{e^{(-n+l)P_{Y}(\widetilde \G)}S_l} (\text {by } (\ref{in1}))\\
\leq \frac{K_2e^C}{K_1} (\text {by Lemma } \ref{step2}).
\end{equation*}
\end{proof}

\begin{lemma}\label{inv}
 Let $\nu$ be the limit of a convergent subsequence $\{\nu_{n_k}\}_{k=1}^{\infty}$
of $\{\nu_n\}_{n=1}^{\infty}$ and let $\mu_n=\frac{1}{n}\sum_{i=0}^{n-1}{\sigma_Y}^{i}(\nu)$.
Then any limit point $\mu$ of $\{\mu_n\}_{n=1}^{\infty}$ is an invariant Gibbs measure for $\widetilde \G$.
\end{lemma}
\begin{proof}
By Lemma \ref{step3}, $\nu$ satisfies, for each cylinder $[i_1\dots i_n]$,
\begin{equation}\label{nugb}
C_1\leq \frac{\nu([i_1 \dots i_n])}{e^{-n P_Y(\widetilde \G)}\tilde g_n(y)}
\leq C_2 , y\in [i_1\dots i_n], n\in \N.
\end{equation}
Now we proceed in a similar way to the proof of Lemma~4.8 in~\cite{Y2}.
Suppose that $\{\mu_{n_k}\}_{k=1}^{\infty}$ converges to $\mu$ in the weak* topology.
To see that $\mu$ is Gibbs, let $i_1\dots i_n$ be a fixed allowable word of length 
$n$ in $Y$. Then for each $l, n\in\N, l>k$,
\begin{align*}
 ({\sigma^{l}_Y}\nu)([i_1\dots i_n])=&\sum_{j_1\dots j_li_1\dots i_n\in B_{l+n}(Y)}\nu([j_1\dots j_li_1\dots i_n])\\
&\geq C_1 e^{-(l+n)P_{Y}(\widetilde \G)}\sum_{j_1\dots j_li_1\dots i_n\in B_{l+n}(Y)}a_{j_1\dots j_li_1\dots i_n}.
\end{align*}
For a cylinder $j_{1}\dots j_{l-k}$, there exists $m_1\dots m_k\in B_{k}(Y)$
such that $j_1\dots j_{l-k} m_1\dots m_k i_{1}\dots i_{n}\in B_{l+n}(X)$.
For fixed $i_1 \dots i_n$ and $j_{1}\dots j_{l-k}$, using (\ref{keypro}), we obtain 
\begin{equation*}
\sum_{j_1\dots j_{l-k}m_1\dots m_ki_1\dots i_n\in B_{n+l}(Y)}a_{j_1\dots j_{l-k}m_1\dots m_ki_1\dots i_n}
\geq \frac{e^{-kP_X(\F)-2C}m}{M^2}a_{j_1\dots j_{l-k}}a_{i_1\dots i_n}. 
\end{equation*}
Summing over all allowable words $j_1\dots j_{l-k}$ of length $(l-k)$ in $Y$, for each fixed $i_1\dots i_n$, we have
$\sum_{j_1\dots j_l i_1\dots i_n\in B_{l+n}(X)}a_{j_1\dots j_l i_1\dots i_n}\geq 
(e^{-kP_X(\F)-2C}mS_{l-k}a_{i_1\dots i_n})/{M^2}$.
Therefore, 
\begin{align*}
&(\sigma^{l}_Y \nu)([i_1\dots i_n])\geq \frac{C_1e^{-kP_X(\F)-2C}m}{M^2}
e^{-(l-k)P_{Y}(\widetilde \G)}e^{-(k+n)P_{Y}(\widetilde \G)}S_{l-k}a_{i_1\dots i_n} \\
&\geq \frac{C_1 e^{-2C-kP_X(\F)-kP_{Y}(\widetilde \G)}m}{M^2 K_2}e^{-nP_Y(\widetilde \G)}a_{i_1\dots i_n}
\geq \frac{C_1e^{-2C-kP_X(\F)-kP_{Y}(\widetilde \G)}m}{C_2 M^2 K_2}\nu[i_1\dots i_n].
\end{align*}
Similarly, for each fixed $i_1\dots i_n\in B_{n}(Y)$,
\begin{align*}
&(\sigma^{l}_Y \nu)([i_1\dots i_n])
\leq C_2 e^{-(n+l)P_{Y}(\widetilde \G)}\sum_{j_1\dots j_l i_1\dots i_n\in B_{l+n}(Y)}a_{j_1\dots j_l i_1\dots i_n}
\\
&\leq C_2e^{-(n+l)P_{Y}(\widetilde \G)}e^CS_{l}a_{i_1\dots i_n}(\text {by using a similar proof of } (\ref{in1}))\\
&\leq \frac{C_2 e^C}{K_1 C_1}\nu[i_1\dots i_n].
\end{align*}
Using arguments similar to those in the final part of the proof of Lemma~4.8 of~\cite{Y2}, we obtain 
$\bar C_1, \bar C_2>0$ such that 
\begin{equation*}
\bar C_1 \leq \frac{\mu[i_1\dots i_n]}{e^{-nP_Y(\widetilde \G)}\tilde g_n(y)}\leq \bar C_2
\text{ for all } n\in \N, y\in [i_1\dots i_n].
\end{equation*} 
Therefore, $\mu$ is an invariant Gibbs measure for $\widetilde \G$. 
\end{proof}

Next we show that $\mu$ in Lemma~\ref{inv} is ergodic. To prove this, we shall 
need the following lemma, which is  similar to Lemma~4.9 in~\cite{Y2}.

\begin{lemma}\label{short}
Let $u_1=i_1\dots i_n$ and $u_2= j_1\dots j_l$, $l,n\in \N$ be allowable words in $Y$ and let $t>n+2k, t\in\N$. 
Then there exists $N$ such that 
\begin{equation*}
\sum_{u_1b_1\dots b_{t-n}u_2\in B_{l+t}(Y)}a_{u_1b_1\dots b_{t-n}u_2}\geq Na_{u_1}a_{u_2}S_{t-n-2k}.
\end{equation*}
\end{lemma}

\begin{proof}
We synthesize the arguments used to prove Lemma~4.9 of~\cite{Y2}. Let $b_{k+1} \dots b_{t-n-k}$ be an allowable words of length $(t-n-2k)$ in $Y$ and 
call it $c$. 
Then there exists ${b_1 \dots b_k}, b_{t-n-k+1} \dots b_{t-n}$ such that 
$u_1 {b_{1}\dots b_k} c {b_{t-n-k+1}\dots b_{t-n}} u_2$ is allowable in $Y$. 
Denote ${b_{1}\dots b_k}$ by $u$ and 
${b_{t-n-k+1}\dots b_{t-n}}$ by $v$. Fix $u_1, u_2, c$ and $v$. By a proof similar to that of (\ref{keypro}), we obtain 
\begin{equation*}
 \sum_{u_1ucvu_2\in B_{l+t}(Y)}\sup\{\sum_{x\in E_{l+t}(u_1ucvu_2)}f_{l+t}(x)\}\geq 
\frac{e^{-kP_X(\F)-2C}m}{M^2}a_{u_1}a_{cvu_2}.
\end{equation*}
Now fix $u_1, u_2, c$. Summing over all allowable words $u, v$ such that $u_1ucvu_2$ is allowable, similar arguments 
to prove (\ref{keypro}) show that 
\begin{equation*}
 \sum_{u_1ucvu_2 \in B_{l+t}(Y)}\sup\{\sum_{x\in E_{l+t}(u_1ucvu_2)}f_{l+t}(x)\}\geq 
(\frac{me^{-kP_X(\F)-2C}}{M^2})^2 a_{u_1}a_{c}a_{u_2}.
\end{equation*}
Summing over all allowable words $u, c, v$ in $Y$ such that $u_1ucvu_2$ is allowable, we obtain 
\begin{equation*}
\sum_{u_1ucvu_2 \in B_{l+t}(Y) }\sup\{\sum_{x\in E_{l+t}(u_1ucvu_2)}f_{l+t}(x)\}
\geq (\frac{me^{-kP_X(\F)-2C}}{M^2})^2a_{u_1}a_{u_2}S_{t-n-2k}.
\end{equation*}
\end{proof}

\begin{lemma} \label{ergodic}
 If $\nu$ is an invariant Gibbs measure for $\widetilde \G$, then $\nu$ is ergodic.
\end{lemma}
\begin{proof}
Employing the same arguments used in the proof of Lemma~4.10 in~\cite{Y2}, we show that there exists $\widetilde C$ such that for each 
$t>n+2k$ and any two cylinder sets  $[i_1,\dots, i_n], [j_1\dots j_l]$ in $Y$, 
$\nu([i_1\dots i_n])\cap \sigma^{-t}_Y([j_1\dots j_l])\geq \widetilde C\nu([i_1\dots i_n])\nu([j_1\dots j_l])$.
Suppose that $\nu$ is an invariant Gibbs measure for $\widetilde \G$ satisfying (\ref{nugb}).
Denote $i_1\dots i_n$ by $u_1$ and $j_1\dots j_l$ by $u_2$. Then, using Lemma \ref{short},
\begin{align*}
&\nu([u_1]\cap \sigma^{-t}_Y[u_2])=\sum_{u_1b_1\dots b_{t-n}u_2\in B_{l+t}(Y)}\nu([u_1 b_1\dots b_{t-n}u_2])\\
&\geq C_1 e^{-(l+t)P_Y(\widetilde \G)}\sum_{u_1b_1\dots b_{t-n}u_2\in B_{l+t}(Y)}a_{u_1 b_1\dots b_{t-n}u_2}\\
&\geq C_1 N a_{u_1}a_{u_2}S_{t-n-2k}e^{-(l+t)P_Y(\widetilde \G)}
=\frac{C_1 Na_{u_1}a_{u_2}S_{t-n-2k}e^{(n-t)P_{Y}(\widetilde \G)}}{e^{(n+l)P_{Y}(\widetilde \G)}}\\
&\geq \frac{C_1 N}{K_2(C_2)^2 e^{2kP_Y(\widetilde \G)}}\nu([i_1\dots i_n])\nu([j_1\dots j_l]) (\text{by Lemma }\ref{step2}
\text{ and } (\ref{nugb})).
\end{align*}
\end{proof}
\noindent \textbf{Proof of Proposition \ref{gibbs1}}
By Lemmas \ref{inv} and \ref{ergodic}, we construct an invariant ergodic Gibbs equilibrium state $\mu$ for $\widetilde \G$. 
Now using the same proof as that of Theorem~5 in~\cite{B2006}, $\mu$ is the unique ergodic invariant measure satisfying the Gibbs property. 
Using the fact that $\H=\{\log \tilde g_ne^{C}\}_{n=1}^{\infty}$ is a subadditive potential, 
the arguments in the proof of
Theorem 5 in~\cite{B2006} show that $\mu$ is the unique equilibrium state for $\widetilde \G$ and it is mixing (also see the proof of 
Proposition~4.11 in~\cite{Y2}).  

\section{The characterizations of images under factor maps using relative pressure}\label{sectionrelativepressure}
In this section  we study the relation between a unique invariant Gibbs measure $\mu_{\F}$ for an almost additive potential 
$\F$ and its image under a factor map $\pi$ with 
connection to relative pressure. The purpose of this section is to prove Theorem \ref{general} and Proposition \ref{bow2}
which characterize $\pi\mu_{\F}$ 
as an equilibrium state for a relative pressure and
$\mu_{\F}$ as a relative equilibrium state. 
Corollaries \ref{level0} and \ref{saturatedcase} are special cases of Theorem \ref{general}.   

We use the relative variational principle to characterize the 
image $\pi\mu$. We continue to use the notation of $\F, \G,\widetilde \G$ from Section \ref{sectionfactor}.

Relative pressure has already been  defined for subadditive potentials (see Section \ref{background}). 
Here we  define relative pressure for almost additive potentials and 
show that the relative variational principle holds for almost additive potentials by simple observations. 

Let $(X, \sigma_X)$ and $(Y, \sigma_Y)$ be subshifts and $\pi:X\rightarrow Y$ be a factor map. 
For an almost additive potential  
$\Phi=\{\log \phi_n\}_{n=1}^{\infty}$ on $X$ satisfying 
$e^{-C}\phi_n(x)\phi_m(\sigma^n_Xx)\leq \phi_{n+m}(x)\leq e^C\phi_n(x)\phi_m(\sigma^{n}_Xx),$
let $\Phi_1=\{\log 
\phi_n e^C\}_{n=1}^{\infty}$. Then $\Phi_1$ is a subadditive potential on $X$. 

Hence for an almost additive potential $\Phi$ on $X$, we define 
for each $y\in Y, n \in\N$, 
$P_n(\sigma_X, \pi, \Phi, \epsilon)(y), 
P(\sigma_X, \pi, \Phi,\epsilon)(y)$ and $P(\sigma_X, \pi, \Phi)(y)$ in same manner as they are defined for 
a subadditive potential (see Page~\pageref{defirp} in Section~\ref{background}). 

It is clear by definition that $P(\sigma_X, \pi, \Phi)(y)=P(\sigma_X, \pi, \Phi_1)(y)$ for all $y\in Y$. 
Since $\lim_{n\rightarrow \infty}({1}/{n})\int \log \phi_n e^C d\mu=\lim_{n\rightarrow\infty}({1}/{n})\int \log \phi_n d\mu$ for all 
$\mu\in M(X, \sigma_X)$, applying the relative variational principle for subadditive potentials, we easily obtain the 
following relative variational principle for almost additive potentials (see Theorem \ref{rvpforsub} in Section~\ref{background}).
   
\begin{theorem}\label{vpaa}
Let $(X,\sigma_X), (Y, \sigma_Y)$ be subshifts, $ \pi:X\rightarrow Y$ a factor map and
$\Phi=\{\log \phi_n\}_{n=1}^{\infty}$ an almost additive potential on $X$.
Then for each $m\in M(Y, \sigma_Y)$,
\begin{equation*}
\int_{Y} P(\sigma_X, \pi, \Phi)dm=\sup\{h_{\mu}(\sigma_X)- h_{m}(\sigma_Y)+\lim_{n\rightarrow \infty}\frac{1}{n}\int_{X} \log \phi_n d\mu : \mu\in M(X, \sigma_X) \text{ and }
\pi\mu=m\}.
\end{equation*}
\end{theorem}
\begin{remark}
Since $\Phi$ is almost additive, there exists $C_2>0$ such that for any $\mu\in M(X,\sigma_X)$, $\vert 
\lim_{n\rightarrow \infty}(1/n)\int\log \phi_{n}d\mu \vert \leq C_2$.
\end{remark}

Before we go further, we review some theorems 
that relate $\mu$ and $\pi\mu$ by using pressure theory.

\begin{theorem}\label{thm1}\cite{Y2}
Let $(X,\sigma_X), (Y, \sigma_Y)$ be subshifts and $\pi:X \rightarrow Y$ a factor map. Define
$F:Y\rightarrow \R$ by $F(y)=P(\sigma_X, \pi, 0)(y)$ and a subadditive potential 
$\Phi_s=\{\log \phi_n\}_{n=1}^{\infty}$ on $Y$.
Define $\Phi_s\circ\pi=\{\log (\phi_n\circ\pi)\}_{n=1}^{\infty}$. Then
\begin{align}
&\sup_{\mu \in M(X, \sigma_X)}\{h_{\mu}(\sigma_X)
-\int F \circ \pi d\mu +\lim_{n\rightarrow\infty}\frac{1}{n} \int \log (\phi_n\circ\pi)d \mu\} \label{gsumu}\\&=\sup_{m\in
M(Y,\sigma_Y)}\{h_{m}(\sigma_Y)+\lim_{n\rightarrow \infty}\frac{1}{n}\int \log \phi_n dm\} \label{gsumd}
\end{align}
for all $\Phi_s$. If $P_Y(\Phi_s)\neq -\infty$, Then $\mu $ is an equilibrium state for $-F
\circ \pi+\Phi_s\circ\pi$ if and only if (i) $\pi\mu$ is an
equilibrium state for $\Phi_s$ and (ii) $\mu$ is a relative equilibrium state for 0 over $\pi\mu$.
\end{theorem}
\begin{remark}
$-F\circ\pi$ above is called a (measurable) compensation function. 
Compensation functions were introduced by Boyle-Tuncel \cite{BT} and 
studied by Walters \cite{W2}. 
\end{remark}

\begin{coro}\label{impcoro}
Theorem \ref{thm1} holds for an almost additive potential $\Phi_s$ on $Y$, and $P_{Y}(\Phi_s)$ is finite.  
\end{coro}
\begin{proof}
Let $\Phi_s=\{\log \phi_n\}_{n=1}^{\infty}$ on $Y$ be an almost additive potential satisfying
$e^{-C}\phi_n(x)\phi_m(\sigma^{n}_X x)\leq \phi_{n+m}(x)\leq e^{C} \phi_n(x)\phi_m(\sigma^{n}_X x)$ for some $C>0$. 
Then $\bar \Phi_s=\{\log e^{C}\phi_n\}_{n=1}^{\infty}$ is a
subadditive potential on $Y$. Replacing $\Phi_s$ in Theorem \ref{thm1} by $\bar \Phi_s$, we obtain 
$\lim_{n\rightarrow\infty}\frac{1}{n} \int \log (e^C\phi_n\circ\pi)d \mu= \lim_{n\rightarrow\infty}\frac{1}{n} \int \log (\phi_n\circ\pi)d \mu$
for all $\mu\in M(X, \sigma_X)$ and $\lim_{n\rightarrow \infty}\frac{1}{n}\int \log e^C\phi_n dm=
\lim_{n\rightarrow \infty}\frac{1}{n}\int \log \phi_n dm$ for all $m\in M(Y, \sigma_Y)$. Therefore, the equality 
in Theorem \ref{thm1} holds. Almost additivity of $\Phi_s$ implies that $P_Y(\Phi_s)$ is finite.  The rest of the theorem holds 
using $M_{-F\circ \pi+\Phi_s\circ \pi}(X, \sigma_X)= M_{-F\circ \pi+\bar\Phi_s\circ \pi}(X, \sigma_X)$ 
and $M_{\Phi_s}(Y, \sigma_Y)= M_{\bar\Phi_s}(Y, \sigma_Y)$, and making the same arguments as in the proof of Theorem \ref{thm1} 
(see Theorem~3.9 in~\cite{Y2}).
\end{proof}

Under the assumption of Theorem \ref{thm1}, for $y_1\dots y_n\in B_n(Y)$, denote by 
$\vert \pi^{-1}[y_1\dots y_n]\vert $ the cardinality of the set 
consisting of exactly one point from each cylinder 
$[x_1 \dots x_n]$ in $X$ such that $\pi([x_1\dots x_n])\subseteq [y_1\dots y_n]$. \label{defipi} 
For $y\in Y$, let  $\tilde \phi_n(y)= \vert \pi^{-1}[y_1\dots y_n]\vert$.
Then $\tilde\Phi=\{\log \tilde \phi_n\}_{n=1}^{\infty}$ is a subadditive potential on $Y$.  We continue to use this notation
throughout this section.  

\begin{theorem}\cite{Y3}\label{thm2}
Suppose in Theorem~\ref{thm1} that $\pi:X \rightarrow Y$ is a factor map between irreducible sofic shifts. 
Define $\tilde \Phi\circ \pi=\{\log (\tilde \phi_n \circ \pi)\}_{n=1}^{\infty}$. Then
$$\lim_{n\rightarrow \infty}\frac{1}{n} \int \log (\tilde \phi_n \circ\pi) d\mu= 
\int F \circ \pi d\mu \text{ for all } \mu \in M(X, \sigma_X).$$ 
Hence we can replace $F$ in Theorem~\ref{thm1} by the subadditive potential $\tilde\Phi$.
\end{theorem}

Now we first consider our question for a simple case. 
By using relative pressure, 
we will characterize the measure $\mu$ of maximal entropy as a relative equilibrium state for $0$ over $\pi \mu$. 
In relation to our work, we note that for an invariant ergodic measure Petersen, Quas and Shin \cite{PQS} and Allahbakhshi and Quas \cite{AQ} 
studied counting the number of preimage measures which have maximal entropy among 
all measures in the fibre.

\begin{coro}\label{level0}
Let $(X,\sigma_X)$, $(Y, \sigma_Y)$ be subshifts and $\pi:X\rightarrow Y$ be a factor map.
Suppose that $X$ has the specification property.
Let $\mu \in M(X, \sigma_X)$ be the unique measure of maximal entropy for $(X,\sigma_X)$ and let $\pi\mu=\nu\in M(Y, \sigma_Y)$.
Then $\mu$ is the unique relative equilibrium state for $0$
over $\nu$.  
\end{coro}

\begin{proof}
We apply Theorems \ref{thm1} and \ref{thm2}. Set $\phi_n=\tilde \phi_n$ for all $n\in \N$ in Theorem \ref{thm1}.
Then $\nu$ is an equilibrium state for 
$\G=\{\log \vert \pi^{-1}[y_1\dots y_n] \vert\}_{n=1}^{\infty}$. Applying Theorem \ref{main1} 
(set $f_n=1$ in Theorem \ref{main1}), it is the unique equilibrium state for $\G$.  Thus Theorem \ref{thm1} implies that
 $\mu$ is a relative equilibrium state for $0$ over $\nu$. Assume that 
there exists $\mu_1\neq \mu$ which is also a relative equilibrium state for $0$  
over $\nu$. Then, by Theorem \ref{thm1}, $\mu_1$ is also a measure of maximal entropy which is a contradiction.   
\end{proof}

Now we want to extend Corollary \ref{level0} for a unique invariant Gibbs measure $\mu$ for an almost additive potential $\F=
\{\log f_n\}_{n=1}^{\infty}$ on a subshift $X$. We observe that we cannot apply Theorem \ref{thm1}, because
in this theorem we only consider a sequence of continuous functions $-F\circ \pi+ \Phi \circ \pi$ on $X$, where  
$\Phi \circ \pi =\{\log (\phi_n \circ \pi)\}_{n=1}^{\infty}$, $\phi_n\in C(Y)$.  

\begin{theorem}\label{general}
Let $(X,\sigma_X)$, $(Y, \sigma_Y)$ be full shifts and $\pi:X\rightarrow Y$ be a factor map. 
Let $\mu_{\F} \in M(X, \sigma_X)$
be a unique invariant Gibbs measure for an almost additive potential $\F=\{\log f_n\}_{n=1}^{\infty}$ on $X$ with
bounded variation.
Let $\pi\mu_{\F}=\nu$. Then $\nu$ is the unique equilibrium state for the relative pressure $P(\sigma_X, \pi, \F)$
and $\mu_{\F}$ is the unique relative equilibrium state for $\F$ over $\nu$.  
\end{theorem}
\begin{remark}
Related work is found in Barral and Feng \cite{BF} in a more general
setting. 
Our result differs slightly from theirs due to our particular setting.
\end{remark}


In order to show Theorem \ref{general}, we need the following simple lemmas. 

\begin{lemma}\label{obvious}
Let $\F=\{\log f_n\}_{n=1}^{\infty}$ be almost additive on a subshift $X$ with bounded variation and define $\F_1=\{\log f_ne^{C}\}_{n=1}
^{\infty}$. For $y\in Y$, let $D_n(y)$ be a set consisting of a point from each cylinder $[x_1\dots x_n]$ such that
$[x_1\dots x_n]\cap \pi^{-1}\{y\}\neq \emptyset$. Then 
$$P(\sigma_X, \pi, \F)(y)= P(\sigma_X, \pi, \F_1)(y)=\limsup_{n\rightarrow\infty}\frac{1}{n}
\log (\sum_{x\in D_n(y)} f_n(x)).$$ 
\end{lemma}
\begin{proof}
The first equality is obvious from the definition of relative pressure. To see the second equality, 
we note that $\F_1$ is subadditive and 
$e^{-2C}(f_n(x)e^{C})(f_{m}(\sigma^{n}_Xx)e^{C})
\leq f_{n+m}(x)e^{C}$. Therefore, the result follows immediately from (\ref{nicef}).  
\end{proof}

\begin{lemma}\label{full}
Let $(X,\sigma_X)$ and $(Y, \sigma_Y)$ be full shifts and 
let $\pi:X\rightarrow Y$ be a factor map. 
For an almost additive potential $\F=\{\log f_n\}_{n=1}^{\infty}$ on $X$ with bounded variation, 
$$P(\sigma_X, \pi, \F)(y)=\limsup_{n\rightarrow\infty}\frac{1}{n}
\log (\sum_{x\in D_n(y)} f_n(x))=\limsup_{n\rightarrow\infty}\frac{1}{n}\log g_n(y)
=\lim_{n\rightarrow\infty}\frac{1}{n}\log g_n(y).$$
\end{lemma}
\begin{proof}
The first equality is obvious from Lemma~\ref{obvious}. Since $X$ is a full shift, given a set $D_n(y)$, we can find
a set $E_n(y)$ such that $E_n(y)=D_n(y)$. Conversely, given  a set $E_n(y)$, we can construct a set $D_n(y)$ such that  $D_n(y)=E_n(y)$. Thus we have the second equality. The third equality is clear because the sequence
$\{\log g_ne^{C}\}_{n=1}^{\infty}$ is a subadditive potential on $Y$.
\end{proof}

\noindent \textbf{Proof of Theorem \ref{general}}\\
The first statement of Theorem \ref{general} is proved by Lemma \ref{full}.
Clearly,
\begin{align*}
&h_{\mu_{\F}}(\sigma_X)+\lim_{n\rightarrow\infty}\frac{1}{n} \int \log f_n d\mu_{\F} =
\sup\{h_{\bar \mu}(\sigma_X)+\lim_{n\rightarrow\infty}\frac{1}{n} \int \log f_n d\bar \mu:\bar \mu\in M(X, \sigma_X)\}\\
&=\sup\{h_m(\sigma_Y)+\int P(\sigma_X, \pi, \F) dm:m\in M(Y, \sigma_Y)\} 
(\text {by Theorem }\ref{main1} \text{ and Lemma } \ref{full}) \\
&=h_{\nu}(\sigma_Y)+\sup\{h_{\bar\mu}(\sigma_X)-h_{\nu}(\sigma_Y)+ 
\lim_{n\rightarrow\infty}\frac{1}{n}\int \log f_n d\bar \mu:\pi\bar \mu=\nu\} (\text{by Theorems } \ref{main1} \text{ and } \ref{vpaa})\\
&=\sup \{h_{\bar \mu}(\sigma_X)+\lim_{n\rightarrow\infty}\frac{1}{n}\int \log f_n d\bar \mu:\pi \bar \mu=\nu\}.
\end{align*}
This proves that $\mu_{\F}$ is a relative equilibrium state for $\F$ over $\nu$. 
To show the uniqueness, 
assume that there exists $\mu_1
\neq \mu_{\F}$ which is a relative equilibrium state for $\F$ over $\nu$. 
Then, using the equations above, $\mu_1$ is also an equilibrium state for $\F$, which is a contradiction.
This completes the proof.

\begin{coro}\label{saturatedcase}
Under the assumptions of Theorem \ref{general}, suppose that $\F=\{\log (\phi_n\circ \pi)\}_{n=1}^{\infty}$, $\phi_n\in C(Y)$ 
for all $n\in\N$. Then $\mu_{\F}$ is the unique relative equilibrium state for $0$ over $\nu$.                                     
\end{coro}
\begin{proof}
Let $\mu\in M(Y, \sigma_Y)$ be fixed. For any $\mu\in M(X, \sigma_X)$ such that $\pi\mu=m$, 
$\lim_{n\rightarrow \infty}\frac{1}{n}\int \log (\phi_n\circ \pi) d\mu= 
\lim_{n\rightarrow \infty}\frac{1}{n}\int \log \phi_n dm$.
Hence we obtain the result.
\end{proof}

In the next proposition,  we will apply Theorem \ref{general} 
in order to study the relation between a unique invariant Gibbs measure for a function $f\in Bow(X)$ and its image under a factor map.
To do this, we use the following lemma by Petersen and Shin \cite{PS}.

\begin{lemma}\cite{PS}\label{impPS}
Let $(X, \sigma_X), (Y,\sigma_Y)$ be irreducible shifts of finite type
and $\pi: X\rightarrow Y$ be a factor map. Let $f_n(x)=e^{f(x)+\dots + f(\sigma^n_Xx)}$.
For each $f\in C(X)$, 
$$P(\sigma_X, \pi, f)(y)=\limsup_{n\rightarrow\infty} \frac{1}{n}\log (\sum_{x\in E_n(y)}f_n(x)),$$ 
almost everywhere with respect to every $m\in M(Y, \sigma_Y)$.
\end{lemma}

\begin{proposition}\label{bow2}
Let $(X, \sigma_X), (Y,\sigma_Y)$ be topological mixing shifts of finite type
and $\pi: X\rightarrow Y$ be a factor map. Suppose $f\in Bow(X)$ and  
let $\mu_{f}$ be a unique invariant Gibbs measure  
for $f$. Then $\pi \mu_{f}$ is the unique equilibrium state for $P(\sigma_X, \pi, f)$ and 
$\mu_{f}$ is the unique relative equilibrium state for $f$ over $\pi\mu$.
\end{proposition} 
\begin{proof}
Let $f_n(x)=e^{f(x)+\dots +f(\sigma^n_Xx)}$. Since $f\in Bow(X)$,  
$\F= \{\log f_n \}_{n=1}^{\infty}$ is an additive sequence with bounded variation, we have 
$P(\sigma_X, \pi, \F)(y)=\limsup_{n\rightarrow\infty}\frac{1}{n}\log 
\big (\sum_{x\in D_n(y)} f_n(x)\big)$ (see (\ref{nicef})). 
Applying Theorem 4.6 of \cite{W2}, we obtain 
$P(\sigma_X, \pi, f)(y)= P(\sigma_X, \pi, \F)(y)$ for all $y\in Y$.
 Let $\bar g_n(y)$ be defined as in Corollary~\ref{maincoro}.
Using Lemma~\ref{impPS} and 
the fact that $\F$ has bounded variation, 
we obtain $P(\sigma_X,\pi, f)(y)=\lim_{n\rightarrow \infty}\frac{1}{n}\log \bar g_n(y)$ 
with respect to every invariant measure on $Y$.    
Therefore,  
we can make similar arguments to those in the proof of Theorem \ref{general}, replacing 
$\lim_{n\rightarrow\infty}\frac{1}{n}\int \log f_n d\bar \mu, \bar \mu\in M(X, \sigma_X)$ and $P(\sigma_X, \pi, \F)$ by 
$\int f d\bar \mu$ and $P(\sigma_X, \pi, f)$ respectively. This proves the proposition. 
\end{proof}

\section{Preimages of Gibbs measures}\label{sectionpreimage}
In the previous sections, we studied the image of a unique invariant Gibbs measure under a factor map. 
In this section, we will consider a preimage of the Gibbs measure for almost additive potential. 
Let $(X,\sigma_X), (Y, \sigma_Y)$ be sofic shifts and 
$\pi:X\rightarrow Y$ be a factor map. Suppose that $X$ has the specification property. For an almost additive potential 
$\Phi_2=\{\log f_n\}_{n=1}^{\infty}$ on $Y$ with 
bounded variation, let $\nu_{\Phi_2}\in M(Y, \sigma_Y)$ be 
the unique invariant Gibbs measure associated to it. 
Now we want to ask the following question.
Is there any Gibbs measure $\mu_{\Phi_1}\in M(X, \sigma_X)$ associated to a sequence of continuous functions $\Phi_1$ on $X$   
such that $\pi\mu_{\Phi_1}=\nu_{\Phi_2}$? 
We will apply Theorems~\ref{main1} and ~\ref{thm1} to study this problem.
  
Answering this question will lead us to examine further when the image of Gibbs measure is a Gibbs measure. In Proposition \ref{small}, 
we study the condition under which the image of 
the Gibbs measure for $f\in Bow(X)$ is the Gibbs measure for a function that belongs to the Bowen class.   

Throughout this section, we use the potential    
$\widetilde\Phi=\{\log \tilde\phi_n \}_{n=1}^{\infty}$ on $Y$,  
where $\tilde\phi_n(y) =\vert \pi^{-1}[y_1\dots y_n]\vert$ for $y=(y_1, \dots,  y_n, \dots)\in Y$.
$\widetilde\Phi$ is a subadditive potential in general. It is an almost additive potential with the following condition. \\

\noindent \textbf{Condition A}\\
Let $n, m\in \N$. There exists $0<D\leq 1$ such that for any $ y_1\dots y_{n+m}\in B_{n+m}(Y)$, we have  
$D\vert \pi^{-1}[y_1\dots y_n]\vert \vert \pi^{-1}[y_{n+1}\dots y_{n+m}]\vert\leq 
\vert \pi^{-1}[y_{1}\dots y_{n+m}]\vert$.
\begin{remark}
 There is an example of a factor map $\pi:X\rightarrow Y$ between subshifts where $(X, \sigma_X)$ is a topologically mixing subshift of 
finite type without satisfying Condition A (see Example~5.6 in~\cite{Y2}). 
\end{remark}

\begin{theorem}\label{preimage}
Let $(X,\sigma_X), (Y, \sigma_Y)$ be sofic shifts, and 
$\pi:X\rightarrow Y$ be a factor map. Suppose that $X$ has the specification property
and Condition A holds. For an almost additive potential $\Phi_2=\{\log f_n\}_{n=1}^{\infty}$ on $Y$ with bounded variation,   
let $\nu_{\Phi_2}\in M(Y, \sigma_Y)$ be a unique invariant Gibbs measure for $\Phi_2$. 
Then $\Phi_1=\{\log ((f_n \circ \pi)/(\tilde \phi_n\circ\pi))\}_{n=1}^{\infty}$ is an almost additive 
potential on $X$ with bounded variation and there exists a unique invariant Gibbs measure $\mu_{\Phi_1}$ for $\Phi_1$ satisfying
$\pi\mu_{\Phi_1}=\nu_{\Phi_2}$. Then   

\begin{equation*}
\sup_{\mu \in M(X, \sigma_X)}\{h_{\mu}(\sigma_X)+\lim_{n\rightarrow\infty}\frac{1}{n}
\int \log \frac{f_n\circ\pi}{\tilde\phi_n\circ \pi} d\mu\}= 
\sup_{m\in M(Y, \sigma_Y)}\{h_{m}(\sigma_Y)+\lim_{n\rightarrow\infty}\frac{1}{n}
\int \log f_n dm\}.
\end{equation*}
\end{theorem} 
\begin{proof} 
For $n\in \N$, let $h_n(x)=(f_n \circ \pi)(x)/(\tilde\phi_n \circ \pi)(x)$. We first show that there exists 
$A>0$ such that $e^{-A}h_n(x)h_m(\sigma^n_X x)\leq h_{n+m}(x)\leq e^{A}h_n(x)h_m(\sigma^{n}_X x).$
Let $x\in X$ and $\pi (x)=y$. Since $\Phi_2$ is almost additive,  there exists $C_2>0$ such that
$e^{-C_2}f_n(y)f_{m}(\sigma^n_Yy)\leq f_{n+m}(y)\leq e^{C_2}f_n(y)f_m(\sigma^{n}_Yy)$. Using Condition A,  
$$ \vert \pi^{-1}[y_{1}\dots y_{n+m}]\vert\leq\vert 
\pi^{-1}[y_1\dots y_n]\vert \vert \pi^{-1}[y_{n+1}\dots y_{n+m}]\vert$$
and the property of the factor map $\pi$, we obtain
\begin{align*}
&h_{n+m}(x)=\frac{f_{n+m}(\pi x)}{\tilde\phi_{n+m}(\pi x)} 
\leq \frac{f_{n}(\pi x)f_m(\sigma^{n}_Y(\pi x))e^{C_2}}{\tilde\phi_{n}(\pi x) \tilde\phi_m(\sigma^n_Y(\pi x))D}\\
&\leq \frac{f_{n}(\pi x)f_m(\pi(\sigma^{n}_X x))e^{C_2}}{\tilde\phi_{n}(\pi x) \tilde\phi_m(\pi(\sigma^n_X x))D}
=h_{n}(x)h_m(\sigma^{n}_X x)\frac{e^{C_2}}{D}.
\end{align*}
Using similar arguments, we have  $h_{n}(x)h_m(\sigma^{n}_Xx)e^{-C_2} \leq h_{n+m}(x)$.
Thus $\Phi_1$ is almost additive. Next we show that $\Phi_1$ has bounded variation. Since $\Phi_2$ has bounded variation, there exists $M_2>0$ such that
$\sup_{n\in \N}\{f_n(y)/f_n(y'): y_i=y'_i, 1\leq i\leq n\}\leq M_2$.
For $x=(x_1, \dots,  x_n, \dots), x'=(x'_{1}, \dots,  x'_{n}, \dots) \in X$, 
where $x_i=x'_i$ for $1\leq i\leq n$, and noting that $\tilde \phi_n$ depends on the first $n$ coordinates of $y\in Y$, we have
$$\frac{h_n(x)}{h_n(x')}=\frac{f_n(\pi x)\tilde\phi_n(\pi x')}{f_n(\pi x')\tilde\phi_n(\pi x)}\leq M_{2}.$$
Therefore, $\Phi_1$ is almost additive with bounded variation and so  
there is a unique invariant Gibbs measure $\mu_{\Phi_1}$ for $\Phi_1$ which is a unique equilibrium state for $\Phi_1$.
Applying Theorems~\ref{thm1} and \ref{thm2},
we obtain the equality in the theorem and  $\pi\mu_{\Phi_1}=\nu_{\Phi_2}$. 
\end{proof}

Now we want to study a preimage of a unique invariant Gibbs measure $\mu_f$ for $f\in Bow(Y)$. 
 Since $f\in Bow (Y)$, let $f_n(y)=e^{f(y)+\dots +f(\sigma^n_Y y)}$ and define $\Phi_2=\{\log f_n\}_{n=1}^{\infty}$.
Then $\Phi_2$ is almost additive with bounded variation. 
Applying Theorem~\ref{preimage}, we immediately obtain the following. 

\begin{coro}\label{gcoro}
Let $(X,\sigma_X), (Y, \sigma_Y)$ be sofic shifts and 
$\pi:X\rightarrow Y$ be a factor map.  Suppose that $X$ has the specification property and Condition A holds. For $f\in Bow(Y)$, 
let $\nu_{f} \in M(Y, \sigma_Y)$ be the unique invariant Gibbs measure for $f$. 
Then there exists $\mu\in M(X, \sigma_X), \pi\mu=\nu_{f}$ such that 
$\mu$ is the unique invariant Gibbs measure for 
$f\circ\pi-\{\log ({\tilde \phi_n}\circ\pi)\}_{n=1}^{\infty}$ on $X$.
\end{coro}


Next we consider a special case of Corollary \ref{gcoro}. If $(X, \sigma_X)$ is a full shift, 
then Condition A is 
satisfied. In this case, we can always find a preimage measure that is a unique invariant Gibbs measure for a function 
that belongs to the Bowen class. 

\begin{lemma}\label{bow}
Let $(X, \sigma_X)$ and $(Y, \sigma_Y)$ be subshifts and  $\pi:X\rightarrow Y$ be a factor map. 
If $f\in Bow (Y)$, then $f\circ\pi\in Bow(X)$. 
\end{lemma}
\begin{proof}
The proof  is straightforward  by using the definition of the Bowen class.
\end{proof}

\begin{coro}(Full shift case) \label{cover1}
Let $(X, \sigma_X), (Y, \sigma_Y)$ be full shifts and $\pi:X\rightarrow Y$ be a factor map. 
Let $\nu_{f} \in M(Y, \sigma_Y)$ be the unique invariant Gibbs  measure for 
$f\in Bow(Y)$. 
Then there exists  $\mu\in M(X, \sigma_X)$ such that $\pi\mu=\nu_{f}$ and  
$\mu$ is the unique invariant Gibbs measure for a function that belongs to the Bowen class.
\end{coro} 
\begin{proof}   
Let $(Y, \sigma_Y)$ be the full shift of $k$ symbols, $1,\dots, k$.
Let $X$ be the full shift on $r_1+\dots +r_k$ symbols $\{a^1_1, \dots a^1_{r_1}, \dots, a^{k}_1, \dots, a^k_{r_k}\}$   and 
$\pi(a^1_i)=1$ for $1\leq i\leq r_1, \dots, \pi(a^k_{i})=k$, for $1\leq i\leq r_k$. 
Let $y=(y_1, \dots, y_n,\dots) \in Y$ and $N_{i}(y_1\dots y_n)$
be the number of times symbol $i$ appears in $y_1\dots y_n$. 
Then $\vert \pi^{-1}[y_1\dots y_n]\vert =r^{N_1(y_1\dots y_n)}_1\dots 
r^{N_k(y_1\dots y_n)}_k$.
Define $g:Y \rightarrow \R$ by $g(y)=\log r_i$ if $y\in [i]$ for $1\leq i\leq k$.
We claim that, for any $\mu\in M(Y, \sigma_Y)$,
$\lim_{n\rightarrow \infty} (1/n)\int \log \tilde \phi_n dm= \int g dm$.

Since $\sum_{j=0}^{n-1}\chi_{[i]}(\sigma^{j}_Yy)=N_{i}(y_1\dots y_n)$, by the ergodic theorem, if $m\in Erg (Y, \sigma_Y)$,
$$\limsup_{n\rightarrow \infty}\frac{1}{n}N_i(y_1\dots y_n)= \limsup_{n\rightarrow \infty}
\sum_{j=0}^{n-1}\chi_{[i]}(\sigma^{j}_Yy)=\int \chi_{[i]}(y)dm$$ for $m$-a.e. $y \in Y$. 
Noting that 
\begin{align*}
&\lim_{n \rightarrow\infty} \frac{1}{n}\int \log \tilde\phi_n dm=
\int \lim_{n\rightarrow \infty} \frac{1}{n}\log \tilde\phi_n dm=
\int \limsup_{n\rightarrow \infty}\frac{1}{n}
\log \tilde\phi_n dm\\
&=\sum_{i=1}^{k}(\log r_i)\int \limsup_{n\rightarrow \infty}\frac{1}{n}N_i(y_1\dots y_n)dm =\sum_{i=1}^{k}(\log r_i)\int\chi_{[i]}(y)dm ,
\end{align*}
we obtain the claim.
Now in Theorem \ref{preimage}, set $f_n(y)=e^{f(y)+\dots +f(\sigma^n_Y y)}$ and define $\Phi_2=\{\log f_n\}_{n=1}^{\infty}$. 
By $f(\sigma^i_Y(\pi(x)))= f(\pi(\sigma^{i}_Xx))$, we obtain 
$$\lim_{n\rightarrow \infty}\frac{1}{n}\int \log (f_n\circ\pi) d\mu=\int f\circ \pi d\mu.$$
Therefore, for any $\bar \mu\in M(X, \sigma_X)$,
$$\lim_{n \rightarrow\infty}\frac{1}{n}\int \log \frac{f_n\circ\pi}{\tilde \phi_n\circ\pi} d\bar \mu=
\lim_{n\rightarrow \infty} \frac{1}{n}\int \log (f_n\circ \pi) d\bar \mu-\lim_{n\rightarrow \infty} \frac{1}{n}\int 
\log (\tilde \phi_n\circ \pi) d\bar \mu=\int f\circ \pi d \bar\mu -\int g\circ\pi d\bar \mu.$$
By Lemma~\ref{bow}, $f\circ\pi-g\circ\pi\in Bow(X)$. By Theorem \ref{preimage}, $\mu$ is the unique invariant Gibbs measure for 
$f\circ\pi-g\circ\pi$.
\end{proof}
\begin{remark}
Walters \cite{W2} studied this problem under a slightly different setting. We note that a slight modification of Theorem 4.1 \cite{W2} implies 
Corollary \ref{cover1}, observing that $-g\circ\pi$ above in the proof is a (continuous) compensation function (see \cite{W2}).\\  
\end{remark}

Suppose that $\mu\in M(X, \sigma_X)$ is a unique invariant Gibbs measure for an almost additive potential 
$\Phi=\{\log f_n\}_{n=1}^{\infty}$ on $X$ 
with bounded variation.
We know from Theorem~\ref{main1} in Section~\ref{sectionfactor}  
that $\nu=\pi\mu$ is the unique invariant Gibbs measure for the asymptotically subadditive potential 
$\G$ on $Y$ where $\G$ is defined as in Theorem~\ref{main1}.
For this $\nu$, under a certain condition, we can find a preimage measure which is a unique invariant Gibbs measure 
$\mu_1$ for an almost additive potential on $X$ (see Theorem \ref{preimage}).
Thus we have two measures $\mu$ and $\mu_1$ that are projected to $\nu$. In this case, what is the relation between $\mu$ and 
$\mu_1$?  In the rest of this section, we consider this question. 
We will apply Theorems~\ref{main1} and~\ref{thm1} to study this problem.

\begin{proposition} \label{key2}
Let $(X, \sigma_X)$ and $(Y, \sigma_Y)$ be sofic shifts and 
$\pi:X\rightarrow Y$ be a factor map.  Suppose that $X$ has the specification property and 
Condition A holds. Let $\F=\{\log f_n\}_{n=1}^{\infty}$ be 
an almost additive potential on $X$ with bounded variation defined in Theorem \ref{main1} . 
Define $g_n$ and $\G$ as in Theorem \ref{main1} and let   
$\Phi_1=\{\log ((g_n\circ\pi)/(\tilde\phi_n \circ\pi))\}_{n=1}^{\infty}$. Then 
\begin{align*}
&\sup_{\mu\in M(X, \sigma_X)}\{h_{\mu}(\sigma_X)+\lim_{n\rightarrow\infty}\frac{1}{n}\int 
\log f_n d\mu\}=
\sup_{m\in M(Y, \sigma_Y)}\{h_{m}(\sigma_Y)+\lim_{n\rightarrow \infty}\frac{1}{n}\int 
\log g_n dm\} \\
&= \sup _{\mu\in M(X, \sigma_X)}\{h_{\mu}(\sigma_X)-\lim_{n\rightarrow\infty}
\frac{1}{n}\int \log (\tilde\phi_n \circ \pi) d\mu+\lim_{n\rightarrow \infty}
\frac{1}{n}\int \log (g_n \circ \pi) 
d\mu\}. 
\end{align*}
Therefore, $P_X(\F)=P_X(\Phi_1)$. Let $\mu_{\F}, \nu,\mu_{\Phi_1}$ be the unique invariant Gibbs measure 
for $\F, \G$, and $\Phi_1$, respectively. Then $\pi\mu_{\F}=\pi\mu_{\Phi_1}=\nu$.
\end{proposition}   
\begin{proof}   
We obtain the first equality from Theorem~\ref{main1}.
For the second equality, using the fact that $\widetilde \H=\{\log \tilde g_ne^C\}_{n=1}^{\infty}$, where 
$\tilde g_n(y)=g_n(y)e^{-nP_X(\F)}$, 
is a subadditive potential on $Y$ and applying Theorem~\ref{thm1}, we obtain
$$\sup\{h_{\mu}(\sigma_X)+\lim_{n\rightarrow \infty}\frac{1}{n}
\int \log \frac{\tilde g_n\circ \pi}{\tilde\phi_n\circ \pi} d\mu\}  
=\sup_{m\in M(Y, \sigma_Y)}\{h_m(\sigma_Y)+\lim_{n\rightarrow \infty}\frac{1}{n}\int \log \tilde g_n dm\}.
$$
We have $\lim_{n\rightarrow \infty}(1/n)
\int \log ({\tilde g_n\circ \pi}/{\tilde\phi_n\circ \pi}) d\mu=\lim_{n\rightarrow \infty}({1}/{n})
\int \log ({g_n\circ \pi}/{\tilde\phi_n\circ \pi}) d\mu-P_X(\F)$ for all $\mu\in M(X, \sigma_X)$ and 
$\lim_{n\rightarrow \infty}({1}/{n})\int \log  \tilde g_n dm=\lim_{n\rightarrow \infty}({1}/{n})\int \log  g_n dm-P_X(\F)$ for all 
$m\in M(Y, \sigma_Y)$. Hence we obtain the second equality. The rest follows by the same proof used to show Theorem~\ref{thm1} 
(see \cite{Y2}).
\end{proof}  

Recall by Theorem \ref{general} that  
$\mu_{\F}$ in Proposition \ref{key2} is the relative equilibrium state of $\F$ over $\pi\mu=\nu$. 
Also, Proposition \ref{key2} implies that $\mu_{\Phi_1}$ is the relative equilibrium state of $0$ over $\nu$ 
(see Theorem~4.3).
Hence, in general, $\mu_{\F}\neq \mu_{\Phi_1}$. When do we have $\mu_{\F}=\mu_{\Phi_1}$? 
The next proposition gives an answer to this question.

\begin{proposition}\label{generalp}
In Proposition \ref{key2},
$\mu_{\F}=\mu_{\Phi_1}$ if and only if 
there exists $A>0$ such that for any $y\in [y_1, \dots, y_n], x\in [x_1, \dots x_n]$ where 
$\pi([x_1\dots x_n])\subseteq [y_1\dots y_n]$, 
\begin{equation}\label{condition2}
\frac{1}{A}\leq \frac{g_n(y)}{\vert \pi^{-1}[y_1 \dots y_n]\vert f_n(x)}\leq A. 
\end{equation}
\end{proposition}
\begin{proof} 
Suppose that $\mu_{\F}=\mu_{\Phi_1}$. For any $n, m\in \N$ there exist $C_1, C_2>0$ such that 
\begin{equation}\label{up1}
 \frac{1}{C_1}\leq \frac{\mu_{\F}([x_1\dots x_n])}{e^{-nP_X(\F)}f_n(x)}\leq C_1 \text{ for all } x\in [x_1\dots x_n]
\end{equation}
and 
\begin{equation}\label{up2}
 \frac{1}{C_2}\leq \frac{\mu_{\Phi_1}([x_1\dots x_n]) \vert \pi^{-1}[y_1\dots y_n]\vert}
{e^{-nP_X(\Phi_1)}g_n(\pi(x))}\leq C_2 \text{ for all } x\in [x_1\dots x_n] 
\text{ such that } \pi([x_1\dots x_n])\subseteq [y_1 \dots y_n].  
\end{equation}
By (\ref{up1}), (\ref{up2}) and $P_{X}(\F)=P_{X}(\Phi_1)$, we obtain
\begin{equation}\label{up3}
\frac{\vert \pi^{-1}[y_1\dots y_n]\vert f_n(x)}{C_1C_2 g_n(\pi(x))}
\leq \frac{\mu_{\F}([x_1\dots x_n])}{\mu_{\Phi_1}([x_1\dots x_n])}\leq
 \frac{C_1C_2\vert \pi^{-1}[y_1\dots y_n]\vert f_n(x)}{g_n(\pi(x))} 
\end{equation}
for any $x\in [x_1\dots x_n]$ such that  $\pi([x_1\dots x_n])
\subseteq [y_1\dots y_n]$. Since 
$\mu_{\F}([x_1\dots x_n])=\mu_{\Phi_1}([x_1\dots x_n])$, (\ref{up3}) implies that 
\begin{equation*}
\frac{1}{C_1C_2}\leq \frac{g_n(\pi(x))}{f_n(x)\vert \pi^{-1}[y_1\dots y_n] \vert}\leq C_1C_2  
\end{equation*}
for any $x\in [x_1\dots x_n]$ such that $\pi([x_1\dots x_n])\subseteq [y_1\dots y_n], n\in\N$.   
Hence we obtain the result. 
Conversely, suppose we have (\ref{condition2}). 
Using (\ref{up3}), 
we obtain
\begin{equation}\label{singular}
\frac{1}{AC_1C_2}\leq \frac{\mu_{\F}([x_1\dots x_n])}{\mu_{\Phi_1}([x_1\dots x_n])}\leq AC_1C_2 
\text { for each  cylinder set } [x_1\dots x_n] \text{ in } X, n\in \N. 
\end{equation}
Since $\mu_{\F}$ and $\mu_{\Phi_1}$ are both ergodic, they are either mutually singular or equal. Using (\ref{singular}), 
they are mutually absolutely continuous. Therefore, $\mu_{\F}=\mu_{\Phi_1}$.
\end{proof}  

Next we consider the special case 
when $\mu_{\F}$ is the unique invariant Gibbs measure for 
$\F=\Phi\circ \pi=\{\log (\phi_n\circ\pi)\}_{n=1}^{\infty}$ on $X$, where $\Phi=\{\log \phi_n\}_{n=1}^{\infty}$ is an almost additive
potential on $Y$ with bounded variation. 

\begin{lemma}\label{aaf}
Let $(X, \sigma_X)$ and $(Y, \sigma_Y)$ be subshifts and let $\pi:X\rightarrow Y$ be a factor map. 
If $\Phi=\{\log \phi_n\}_{n=1}^{\infty}$ is an almost additive potential on 
$Y$ with bounded variation, then 
$\Phi\circ \pi$ is an almost additive 
potential on $X$ with bounded variation. 
\end{lemma}
\begin{proof}
This is a generalization of Lemma \ref{bow}. The proof is immediate by using the definitions of almost additivity and bounded variation.  
\end{proof}

\begin{coro}\label{special}
In Proposition \ref{key2}, let $\F=\Phi\circ \pi=\{\log (f_n\circ\pi)\}_{n=1}^{\infty}$, 
where $\Phi=\{\log f_n\}_{n=1}^{\infty}$ is an almost additive
potential on $Y$ with bounded variation. Then $\mu_{\F}=\mu_{\Phi_1}$.
\end{coro}
\begin{proof}
We use Proposition \ref{generalp}. Recall the definition of $g_n$ from Section \ref{sectionfactor}.
Fix $y_1\dots y_n\in B_n(Y)$. 
\begin{align*}
\vert \pi^{-1}[y_1\dots y_n]\vert \inf_{\pi(x)\in [y_1\dots y_n]}(f_n\circ \pi) (x)
\leq g_n(y) \leq \vert \pi^{-1}[y_1\dots y_n]\vert \sup_{\pi(x)\in [y_1\dots y_n]} (f_n\circ \pi)(x).
\end{align*} 
Let $y\in [y_1\dots y_n]$. Then for any $x\in [x_1\dots x_n]$ such that $\pi([x_1\dots x_n])\subseteq [y_1\dots y_n]$,
\begin{equation*}
 \frac{\inf_{\pi(x)\in [y_1\dots y_n]}(f_n\circ \pi)(x)}{(f_n\circ\pi)(x)}\leq \frac{g_n(y)}
{\vert \pi^{-1}[y_1\dots y_n]\vert (f_n\circ\pi)(x)}\leq
\frac{\sup _{\pi(x)\in [y_1\dots y_n]}(f_n\circ \pi)(x)}{(f_n\circ\pi)(x)}.
\end{equation*}
Since $\Phi$ has bounded variation, there exists $M>0$ such that $f_n(y')/f_n(y)\leq M$ for 
$y, y'\in Y, y_i=y'_i$ for $1\leq i\leq n$. 
Therefore, for any $x\in [x_1\dots x_n]$ such that $\pi([x_1\dots x_n])\subseteq [y_1\dots y_n]$,
\begin{equation*}
\frac{1}{M}\leq \frac{g_n(y)}{\vert \pi^{-1}[y_1\dots y_n]\vert (f_n\circ\pi)(x)}\leq
M.
\end{equation*}
\end{proof}

In Corollary \ref{cover1}, we considered a preimage of a unique invariant 
Gibbs measure for $f\in Bow(Y)$ and 
showed that we can find a preimage
which is a unique invariant Gibbs measure for a function that belongs to the Bowen class. 
To see this, we studied conditions under which
$\Phi_1$ in Proposition \ref{key2} can be replaced by a continuous function that belongs to the Bowen class, and 
we chose the condition
of $(X, \sigma_X)$ being a full shift.  Now given a unique invariant 
Gibbs measure $\nu$ for $f\in Bow(Y)$, we 
consider another preimage $\mu_{\F}$ by using 
a unique invariant Gibbs measure $\mu_{\F}$ for $\F$ in Proposition \ref{key2}. We will show that we can replace $\F$
by a function that belongs to the Bowen class when $(X, \sigma_X)$ is a sofic shift with the specification property.  

\begin{coro}\label{co2}
Let $(X, \sigma_X)$, $(Y, \sigma_Y)$ be sofic shifts and 
$\pi:X\rightarrow Y$ be a factor map.    Suppose that $X$ has the specification property and Condition A holds.
Let $f \in Bow (Y)$ and $\nu\in M(Y, \sigma_Y)$ be 
the unique invariant Gibbs measure for $f$. 
Let $\mu_{f\circ\pi}$ be 
the unique Gibbs measure for $f\circ \pi\in Bow(X)$.
Define $\Phi_1=\{\log ((f_n\circ\pi)/(\tilde \phi_n\circ\pi))\}_{n=1}^{\infty}$, where 
$f_n(y)=e^{f(y)+\dots +f(\sigma^n_Yy)}$, and $\mu_{\Phi_1}$ as in Theorem \ref{preimage}. 
Then $\mu_{f\circ\pi}=\mu_{\Phi_1}$.
\end{coro}
\begin{proof}
Let $\Phi=\{\log f_n\}_{n=1}^{\infty}$, where $f_n(y)=e^{f(y)+\dots +f(\sigma^n_Y y)}$
and $\F=\Phi\circ\pi$. Applying Corollary~\ref{special}, $\mu_{\F}=\mu_{\Phi_1}$. Since 
$\mu_{\F}=\mu_{f\circ \pi}$, we obtain the result.  
\end{proof}

Let $f\in C(Y)$.   
In Corollary \ref{co2}, a unique invariant Gibbs measure $\mu_{f\circ\pi}$ for $f\circ \pi$, where $f\in Bow(Y)$,  is 
projected to a unique invariant Gibbs measure $\nu$ for $f$. 
In the next proposition we consider a slightly more general condition of $\phi\in Bow(X)$ 
under which a unique invariant Gibbs measure $\mu_{\phi}$ for $\phi$ is projected to a unique invariant Gibbs measure
$\nu$ for a function that belongs to the Bowen class. The question of images of Gibbs measures for functions of summable variation has recently 
been  
studied for full shifts \cite{CU, PK} and for subshifts \cite{K, JY}. 

The next theorem implies, for example, 
that if $f$ is in the Bowen class and has bounded variation in the preimages of all cylinders  
$[y_1\dots y_n]$ in $Y$, i.e., there exists $M>0$ for any $y_1 \dots y_n \in B_n(Y)$ such that 
$\sup_{n\in\N}\{e^{f(x)+\dots +f(\sigma^n_Xx)}/e^{f(x')+ \dots+ f(\sigma^n_X x')}
: \pi(x_1\dots x_n)=\pi(x'_1\dots x'_n)=y_1 \dots y_n\}\leq M$, then the projection $\nu$ is a unique invariant Gibbs measure for a function
that belongs to the Bowen class. 

\begin{proposition}\label{small}
Let $(X, \sigma_X)$ and $(Y, \sigma_Y)$ be full shifts and $\pi:X \rightarrow Y$  a factor map.
Let $f \in Bow(X)$ and $\mu_{f}$ be a unique invariant Gibbs measure 
for $f$. Let $f_n(x)=e^{f(x)+\dots+f(\sigma^{n-1}_Xx)}$. For $n\in \N$, define $\bar g_n$ and 
$\bar \G=\{\log \bar g_n\}_{n=1}^{\infty}$ on $Y$ 
as in Corollary~\ref{maincoro}. 
Suppose that there exists $A>0$ such that  
$$\frac{1}{A}\leq \frac{\bar g_n(y)}
{\vert \pi^{-1}[y_1\dots y_n]\vert f_n (x)}\leq A,$$
for each $x\in [x_1\dots x_n]$ such that $\pi([x_1\dots x_n])\subseteq [y_1\dots y_n]$.
Then $\nu=\pi\mu_{f}$ is a unique invariant Gibbs measure 
for a continuous function on $Y$ which belongs to the Bowen class.
\end{proposition}
\begin{proof}
Let $(X, \sigma_X)$ be the full shift of $k$ symbols, $1, \dots, k$.  First recall from Theorem \ref{main1} that 
$\nu$ is the unique invariant Gibbs measure for $\bar \G=\{\log \bar g_n\}_{n=1}^{\infty}$.  
For $y\in Y$, define $\psi:Y\rightarrow X$ by $\psi(y)=x$, where $x=(x_1, \dots, x_n, \dots)$ is defined by $\pi(x)=y$ with the property that 
for any $x'=(x'_1, \dots, x'_n, \dots)
\in X$ such that $\pi(x')=y, 1\leq x_i\leq x'_i$ for all $i\in \N$. 
In other words, for each $x_i$, we take $x_i$ to be the smallest positive integer projected to $y_i$. 
Such $x$ is uniquely determined. 
Clearly $\psi$ is a continuous function on $Y$. 
By assumption, for each $x$ such that $\pi (x)\in [y_1\dots y_n], y\in [y_1\dots y_n]$,  
$$\bar g_n(y)\leq A
\vert \pi^{-1}[y_1\dots y_n]\vert
e^{f(x)+\dots+f(\sigma^{n-1}_X(x))}.$$
Hence for each $y\in [y_1\dots y_n]$,
\begin{equation}\label{ine1}
\bar g_n(y) \leq A\vert \pi^{-1}[y_1\dots y_n]\vert
e^{f(\psi(y))+
f(\sigma_{X}(\psi (y)))+\dots+f(\sigma^{n-1}_X (\psi(y)))}.
\end{equation}
Similarly, we have
\begin{equation} \label{ine2}
\bar g_n(y)\geq 
\frac{1}{A}\vert \pi^{-1}[y_1\dots y_n] \vert
e^{f(\psi(y))+
f(\sigma_{X}(\psi(y)))+\dots +f(\sigma^{n-1}_X(\psi(y)))}.
\end{equation}
We claim that $f\circ\psi \in  Bow(Y)$.
Since $f\in Bow(X)$, there exists $M$ such that 
$$\sup\{(f(x)+\dots+f(\sigma^{n-1}_Xx))-(f(x')+\dots +f(\sigma^{n-1}_Xx')):
x_i=x'_i, 1\leq i\leq n\}\leq M, \text {for all } n\in \N.$$
For $y,y'\in Y$, $y_i=y'_i$ for $1\leq i\leq n,$ $\psi(y)_i=\psi(y')_i$
for $1\leq i\leq n$ by definition of $\psi$. Also, clearly, 
$\psi(\sigma^{i}_Yy)=\sigma^{i}_X(\psi(y))$ for each $i=0, 1, \dots, n-1.$
Therefore, 
$$\sup\{(f(\psi(y))+\dots+f(\psi(\sigma^{n-1}_Yy)))-
(f(\psi(y'))+\dots +f(\psi(\sigma^{n-1}_Yy')):
y_i=y'_i, 1\leq i\leq n\}\leq M, \text {for all } n\in \N.$$
Hence, $f\circ\psi\in Bow(Y)$. For any $m\in M(Y, \sigma_Y)$, using (\ref{ine1}),  
\begin{align*}
&\lim_{n\rightarrow \infty}\frac{1}{n}\int \log \bar g_n(y)dm
\leq \lim_{n\rightarrow \infty}\frac{1}{n}\int \log (A \vert \pi^{-1}[y_1\dots y_n]\vert e^{f(\psi(y))+\dots 
+f(\sigma^{n-1}_Y\psi(y))}) dm\\ 
&=\lim_{n\rightarrow\infty} \frac{1}{n}\int \log \vert \pi^{-1}[y_1\dots y_n]\vert dm+ 
\lim_{n\rightarrow\infty}\frac{1}{n}\int \log e^{f(\psi(y))+\dots 
+f(\sigma^{n-1}_Y\psi(y))}dm\\
&=\int (g +f\circ \psi) dm,
\end{align*}
where $g$ is a locally constant function defined in the proof of Corollary \ref{cover1}.
Similarly, using (\ref{ine2}), we obtain  
$\int (g +f\circ \psi) dm\leq \lim_{n\rightarrow \infty}\frac{1}{n}\int \log \bar g_n(y)dm$. 
Therefore, applying Theorem \ref{main1}, $\nu=\pi\mu_{f}$ is the unique Gibbs measure for $g+f\circ \psi\in Bow(Y)$.
\end{proof}

\section{images of Gibbs states for almost additive potentials and continuous functions}\label{sectioncomp}
Let $(X, \sigma_X), (Y, \sigma_Y))$ be subshifts and 
$\pi:X\rightarrow Y$ be a factor map. Suppose that $X$ has the specification property.
Let $\mu$ be a unique invariant Gibbs measure for an almost additive 
potential $\F=\{\log f_n\}_{n=1}^{\infty}$ with bounded variation. 
In Section~\ref{sectionfactor}, we 
characterized $\pi\mu$ as a unique invariant Gibbs measure for an asymptotically subadditive
potential $\G=\{\log g_n\}_{n=1}^{\infty}$ (see Theorem~\ref{main1}). In this section, 
we relate our results to the theory of 
factor maps of Gibbs equilibrium states for continuous functions (see \cite{CUO, CU, V, JY, PK, K}).
   
We start with the case when $\mu$ is a unique invariant Gibbs measure for a continuous function 
which depends only on the 
first coordinate. For this purpose, we define for $n\in \N, x\in X$, 
$f_n(x)=e^{f(x)+\dots +f(\sigma^{n-1}_Xx)}$. 
Then $\F=\{\log f_n\}_{n=1}^{\infty}$ is an additive potential with bounded variation.      
We continue to use this notation throughout this section.

\begin{proposition}
Suppose $(X, \sigma_X)$ is a full shift and 
$f\in C(X)$ depends on the first coordinate of $x\in X$. For $n\in \N$, define $\bar g_n$ and $\bar \G=\{\log \bar g_n\}_{n=1}^{\infty}$ 
on $Y$ as in Corollary~\ref{maincoro}. 
Then, for all $n\geq 2$, 
$\bar g_n(y)=\bar g_1(y)\bar g_{n-1}(\sigma_Yy)$.  Therefore, for all $m\in M(Y, \sigma_Y)$,  
\begin{equation}\label{key3}
\lim_{n\rightarrow \infty}\frac{1}{n}\int \log \bar g_n (y) dm
=\int \log \bar g_1(y) dm.
\end{equation}
Hence the unique invariant Gibbs measure for $\bar \G$ on $Y$, which is the image of the unique invariant 
Gibbs measure for $f$, is the unique invariant Gibbs measure for the 
locally constant function $\log \bar g_1$ on $Y$.   
\end{proposition}
\begin{proof}
Fix $y=(y_1, \dots, y_n,\dots)\in Y$. Let $\pi^{-1}\{y_1\}=\{a^{1}_1, \dots a^{1}_k\}$ for some $k\in \N$.
Take one point $\bar x$ from a cylinder set $[x_2\dots x_n]$ such that  
$[x_2 \dots x_n]\subseteq \pi^{-1}([y_2\dots y_n])$.
Let $\tilde x_i=a^{1}_i\bar{x},1\leq i\leq k$.  
Then $\tilde x_i\in [a^{1}_ix_2\dots x_n]$ and  
$[a^{1}_ix_2\dots x_n]\subseteq \pi^{-1}([y_1\dots y_n])$ 
for each $1\leq i \leq k$. Let $w \in X$ be fixed.
Observe that 
\begin{equation*}
\sum_{i=1}^{k}e^{f(a^{1}_iw)}\sum_{\bar x\in [x_2\dots x_n], 
\pi(x_2\dots x_n)=y_2\dots y_n}
e^{f(\bar x)+\dots +f(\sigma^{n-2}_X \bar x)}
=\sum_{\tilde x_i=a^1_i\bar x, 1
\leq i \leq k}
e^{f(\tilde {x}_i)+\dots +f(\sigma^{n-1}_X \tilde {x}_i)},
\end{equation*}
where the second summation is taken over all $[x_2\dots x_n]$ such that
$[x_2 \dots x_n]\subseteq \pi^{-1}([y_2\dots y_n])$ 
and $\bar x$ in the third summation is taken over the same set as $\bar x$ in the second summation.
Since $f$ depends on the first coordinate, we obtain the above inequality for any $w\in X$ and
for any $\bar x \in [x_2\dots x_n]$. Therefore,  we obtain 
$\bar g_1(y)\bar g_{n-1}(\sigma_Yy)\leq \bar g_{n}(y)$. 
Since   
$\bar \G=\{\log \bar g_n\}_{n=1}^{\infty}$ is a subadditive potential, we obtain 
$\bar g_{1}(y)\bar g_{n-1}(\sigma_Yy)= \bar g_{n}(y)$ for each $y\in Y, n, m\in \N$.
Therefore, we obtain  (\ref{key3}). The rest of the theorem follows immediately from Corollary~\ref{maincoro}.
\end{proof}


Now we consider the special case when $(X, \sigma_X)$ is a full shift and $f\in C(X)$ is of summable variation.
Let $\mu_f\in M(X, \sigma_X)$ be a unique invariant Gibbs measure for $f$. 
Pollicott and Kempton \cite{PK} considered the image of $\mu_f$ under a factor map 
$\pi:X\rightarrow Y$.
Fix $w\in X$. For $n\in \N, y=(y_1,\dots,y_n, \dots)\in Y$, they defined
\begin{equation*}\label{pk1}
g_n(y, w)=\sum_{b_n=x_1\dots x_n\in B_n(X), \pi(x_1\dots x_n)=y_1\dots y_n}
e^{f(b_nw)+\dots+f(\sigma^{n-1}_X b_nw)}
\text{ and } u_{w,n}(y)=\frac{g_{n+1}(y, w)}{g_n(\sigma_Y y, w)}.
\end{equation*}
In particular, they showed that $\{u_{w,n}(y)\}_{n=1}^{\infty}$ converges uniformly to a continuous 
function $u:Y\rightarrow \R$ and it is independent of $w\in X$ (see Proposition~3.2 in~\cite{PK}).
Using this, they proved that $\pi\mu_f$ is an invariant Gibbs measure for $\log u$.
Kempton \cite{K} extended this result to a function $f$ which is of summable variation on a subshift of finite type
under a certain condition. 

In Corollary \ref{maincoro} in Section \ref{sectionfactor}, we saw that  
$\pi\mu_f$ is a unique invariant
Gibbs measure for the subadditive potential $\bar \G$ on $Y$ where $\bar \G$ is defined in Corollary \ref{maincoro}. In the next proposition,  we will see that our construction of $\G$ 
in Theorem \ref{main1} is a generalization of $\log u$ defined in
\cite{PK}. We continue to use $g_n(y, w), u_{w,n}(y)$ and  $u:Y\rightarrow \R$ as above.     
  
\begin{proposition}\label{svcase}
Let $(X, \sigma_X)$ be a full shift and $\pi:X\rightarrow Y$ be a factor map. 
Suppose that $f\in C(X)$  
is of summable variation. Then for all $m\in M(Y, \sigma_Y),$
\begin{equation*}
 \lim_{n\rightarrow \infty}\frac{1}{n}\int \log \bar g_n(y) dm= \int \log u(y) dm,
\end{equation*}
 where $\bar g_n$ is defined as in Corollary \ref{maincoro}. 
\end{proposition}
\begin{proof} 
Fix $w\in X$ and, for any 
$n\in \N, y\in Y$, let $u_{w, n}(y)={g_{n+1}(y, w)}/{g_n(\sigma_Yy, w)}$.
Since $f\in C(X)$ is of summable variation, $\F=\{\log f_n\}_{n=1}^{\infty}$ has bounded variation.  
We first claim that 
\begin{equation}\label{eq1}
\lim_{n\rightarrow \infty}\int 
\frac{1}{n}\log \bar g_{n}(y) dm =\lim_{n\rightarrow \infty}\int 
\frac{1}{n}\log g_{n+1}(y,w) dm.
\end{equation}
Observing that $g_n(y, w)\leq \bar g_n(y)\leq M g_n(y, w)$ 
for each $n\in \N, w \in X$, (\ref{eq1}) is clear.
Next we claim that 
\begin{equation}\label{key4}
1<u_{w, n}(y)\leq M\frac{\bar g_{n+1}(y)}{\bar g_n(\sigma_Yy)}\leq M\bar g_1(y).
\end{equation}
Let $y=(y_1,\dots, y_{n}, \dots)\in Y$. Let $\pi^{-1}\{y_1\}=
\{a^1_1, \dots a^1_{k}\}$ for some $k\in \N$.
Taking any $x_2\dots x_{n+1}\in B_n(X)$ such that 
$\pi(x_2\dots x_{n+1})=y_2\dots y_{n+1}$, 
we have, for $1\leq i\leq k$, $a^{1}_ix_2\dots x_{n+1}\in B_{n+1}(X)$ and 
$\pi(a^{1}_ix_2\dots x_{n+1})=y_1\dots y_{n+1}$. Therefore, clearly, 
$g_{n+1}(y, w)>g_n(\sigma_Yy, w)$.
This implies $1<u_{w, n}(y)$. Also, again using 
the fact that $g_n(y, w)\leq \bar g_n(y)\leq M  g_n(y, w)$, we obtain
the second inequality in (\ref{key4}). The third inequality in (\ref{key4}) is clear
because $\bar \G$ is subadditive. Hence we obtain the claim. Now write $g_{n+1}(y, w)=u_{w, n}(y)g_n(\sigma_Y y, w)$. Then by definition,
\begin{align*}
g_{n+1}(y, w)&=u_{w, n} (y) u_{w, n-1}(\sigma_Y y)g_{n-1}(\sigma^{2}_Y y, w)\\
&=u_{w, n} (y) u_{w, n-1}(\sigma_Y y)u_{w, n-2}(\sigma^2_Y y)g_{n-2}(\sigma^{3}_Y y, w)\\
&=u_{w, n} (y) u_{w, n-1}(\sigma_Y y)\dots u_{w, 1}(\sigma^{n-1}_Y y) g_1(\sigma^{n}_Y y, w).\\
\end{align*}
Fix $n\in\N$. Taking the logarithm of both sides, dividing both sides by $n$ and integrating both sides with respect to $m\in M(Y, \sigma_Y)$,
\begin{align*}
\int \frac{1}{n} \log g_{n+1}(y, w) dm 
&=\int \frac{1}{n}\log (u_{w, n}(y)\dots  u_{w, 1}(\sigma^{n-1}_Y y))dm
 +\int \frac{1}{n} \log g_1(\sigma^{n}_Y y, w)dm\\
&=\int \frac{1}{n}\log (u_{w, n}(y)\dots  u_{w, 1}(y))dm +\int \frac{1}{n} \log g_1(\sigma^{n}_Y y , w)dm,
\end{align*}
where in the second equality we use the fact that $m$ is invariant.
Noting that $g_1(y , w)$ is continuous on $Y$ and thus bounded,  
letting $n\rightarrow \infty$, 
\begin{equation}\label{eq0}
\lim_{n\rightarrow \infty}\int 
\frac{1}{n}\log g_{n+1}(y,w) dm = \lim_{n\rightarrow \infty} \int \frac{1}{n}\log (u_{w, n}(y)\dots  u_{w, 1}(y))dm. 
\end{equation}
We claim that
\begin{equation}\label{eq3}
\lim_{n\rightarrow \infty} \int \frac{1}{n}\log (u_{w, n}(y)\dots  
u_{w, 1}(y))dm=\int \log u(y)dm.
\end{equation}
To see this, we first note by (\ref{key4}) that 
$\vert \log u_{w, n}(y)\vert\leq \log M\max_{y\in Y}\bar{g}_1(y)$ for all $n\in \N$. Let $A=
\max_{y\in Y}\bar{g}_1(y)$. For each $y\in Y, n\in \N$, let $p_{w,n}(y)=(\log (u_{w, n}(y)\dots  u_{w, 1}(y)))/n$. 
Since $\{u_{w,n}\}_{n=1}^{\infty}$ converges
uniformly to a continuous function $u_{w}$ (see \cite{PK}), 
given $\epsilon>0$, there exists $N_w\in \N$ such that 
\begin{equation*}
u_{w}(y)-\epsilon <u_{w,n}(y)<u_{w}(y)+\epsilon \text{ for all } n>N_w, y\in Y.
\end{equation*}
Since $u_{w,n}(y)> 1$, we obtain 
\begin{equation*}
\frac{n-N_w}{n} \log (u_{w}(y)-\epsilon) 
<\frac{1}{n} \log (u_{w,n}(y)\dots u_{w,1}(y))\leq \frac{N_w}{n}\log MA 
+\frac{n-N_w}{n}\log (u_{w}(y)+\epsilon). 
\end{equation*}
Letting $n\rightarrow \infty$, 
\begin{equation*}
\log (u_{w}(y)-\epsilon) \leq \lim_{n\rightarrow\infty}\frac{1}{n} 
\log (u_{w,n}(y)\dots u_{w,1}(y))\leq 
\log (u_{w}(y)+\epsilon). 
\end{equation*}
Since we can do this argument for each $\epsilon >0$, we obtain
\begin{equation}\label{lim}
\lim_{n\rightarrow\infty} p_{w,n}(y)=
\lim_{n\rightarrow\infty} \frac{1}{n} \log (u_{w,n}(y)\dots u_{w,1}(y))=
\log u_{w}(y). 
\end{equation}
Since $p_{w,n}$ is bounded uniformly, using (\ref{lim}), we obtain
\begin{equation}\label{eq4}
\lim_{n\rightarrow \infty} \int \frac{1}{n}\log (u_{w, n}(y)\dots  
u_{w, 1}(y))dm=\int \log u_w(y)dm.
\end{equation}
Since the convergence of $u_{w,n}$ is independent of $w$ (see \cite{PK}), 
we can replace $u_w$ in (\ref{eq4}) by $u$. Hence we obtain (\ref{eq3}).  
The proof is completed by combining
(\ref{eq0}), (\ref{eq1}) and (\ref{eq3}).  
\end{proof}

\section{Question} 
Here we pose a question that relates our results to the theory of factor maps of Gibbs
equilibrium states for continuous functions.

In Theorem \ref{main1}, when can we replace 
$\G$ by a continuous function, i.e., when do we have a $g\in C(Y)$ that satisfies 
$\lim_{n\rightarrow\infty}\frac{1}{n}\int \log g_n dm=\int g dm$ 
for every $m\in M(Y, \sigma_Y)$?\\

{\em Acknowledgements.}
I would like to thank Professor Karl Petersen and Professor Edgardo Ugalde for useful comments on the paper.  
I am also grateful to Professor Markos Maniatis for useful discussion about renormalization.
Finally, I would also like to thank the editor and referee for 
their suggestions and comments which have greatly improved the paper. 
This research was supported by the Center of Dynamical Systems and Related Fields c\'odigo ACT1103  PIA - Conicyt 
and by Proyecto Fondecyt Iniciaci\'{o}n 11110543.

\end{document}